\documentclass[11pt, reqno]{amsart}
\usepackage[top=1.2in,bottom=1.2in,left=1.2in,right=1.2in]{geometry}

\usepackage{float, subcaption}

\usepackage{graphicx, tikz}

\usepackage{amsmath,amssymb,amsthm}

\usepackage[all]{xy}
\usepackage[mathcal]{eucal}
\usepackage{verbatim}
\usepackage[all]{xy}

\theoremstyle{plain}
\newtheorem{theorem}{Theorem}[section]
\newtheorem{lemma}[theorem]{Lemma}
\newtheorem{definition}[theorem]{Definition}
\newtheorem{proposition}[theorem]{Proposition}
\newtheorem{corollary}[theorem]{Corollary}
\newtheorem{claim}[theorem]{Claim}

\newtheorem{fact}[theorem]{Fact}

\newcommand{\nc}{\newcommand}
\nc{\dmo}{\DeclareMathOperator}

\dmo{\Conf}{Conf}
\dmo{\Mod}{Mod}
\dmo{\PMod}{PMod}
\dmo{\Fix}{Fix}
\dmo{\Diff}{Diff}
\dmo{\C}{\mathbb{C}}
\dmo{\Z}{\mathbb{Z}}
\dmo{\CRS}{\mathcal{C}}

\nc{\F}{\mathcal F}
\nc{\G}{\mathcal G}
\dmo{\PMF}{PMF}

\nc{\lei}[1]{{\color{red} \sf  L: [#1]}}
\nc{\aru}[1]{{\color{blue} \sf  A: [#1]}}
\nc{\margin}[1]{\marginpar{\scriptsize #1}}

\usepackage{mathtools}
\DeclarePairedDelimiter\floor{\lfloor}{\rfloor}

\nc{\p}[1]{\medskip\noindent\textbf{#1}}

\usepackage[foot]{amsaddr}

\title{From braid groups to mapping class groups}

\author{Lei Chen$^{1}$, Aru Mukherjea$^{2}$}
\address{$^1$ Department of Mathematics, University of Maryland College Park}
\address{\ \ \ 4176 Campus Dr, College Park, MD 20742}
\address{$^2$ Department of Mathematics, University of Texas Austin}
\address{\ \ \ 2515 Speedway, Austin, TX 78712
}

\email{chenlei@umd.edu, amukherj@utexas.edu}

\begin{document}

\maketitle

\begin{abstract}
In this paper, we classify homomorphisms from the braid group on $n$ strands to the mapping class group of a genus $g$ surface. In particular, we show that when $g<n-2$, all representations are either cyclic or standard. Our result is sharp in the sense that when $g=n-2$, a generalization of the hyperelliptic representation appears, which is not cyclic or standard. This gives a classification of surface bundles over the configuration space of the complex plane. As a corollary, we partially recover the result of Aramayona--Souto \cite{AS}, which classifies homomorphisms between mapping class groups, with a slight improvement.

\bigskip\bigskip

\end{abstract}

\section{Introduction}

Let $B_n:=\pi_1(\Conf_n(\C))$ be the \emph{braid group} on $n$ strands, where $\Conf_n(\C)$ is the \emph{unordered configuration space}, or the space of unordered $n$-tuples of points on $\C$.
For $S_{g,p}$ the surface of genus $g$ with $p$ punctures, let $\PMod(S_{g,p}) := \pi_0(\Diff^+(S_{g,p}))$ be the \emph{pure mapping class group}, where $\Diff^+(S_{g,p})$ is the group of orientation-preserving diffeomorphisms on $S_{g,p}$ fixing the $p$ punctures pointwise. 

A \emph{geometric representation} of a group $G$ is any homomorphism $G\to \PMod(S_{g,p})$. A natural problem is to classify all geometric representations of the braid group $B_n$. Since $\Conf_n(\mathbb{C}) = K(B_n,1)$ and $K(\PMod(S_{g,p}),1)$ is the classifying space for $S_{g,p}$-bundles, we know that classifying the geometric representations of $B_n$ is equivalent to classifying $S_{g,p}$-bundles over $\Conf_n(\C)$.

As $B_n$ is the mapping class group of a punctured sphere, classifying geometric representations of $B_n$ is a special case of a natural question: what are all the homomorphisms between mapping class groups? Mirzakhani conjectured that all injective homomorphisms between mapping class groups should come from ``some manipulations of surfaces", such as embeddings of surfaces and covering constructions (e.g., see \cite{AS}). Castel \cite{Castel} gives a complete classification of homomorphisms $B_n \to \PMod(S_{g,p})$ for $n\ge 6$ and $g\le \frac{n}{2}$. In this paper, we give a new proof of his result and also extend the classification to the case $g<n-2$ but in a reduced range, $n\ge 23$.

We denote the standard generating set of $B_n$ as $\{\sigma_i \mid i=1,2,...,n-1\}$. Let $i(a,b)$ denote the geometric intersection number of two curves $a,b$. We call $\{c_1,...,c_k\}$ a \emph{chain} of simple closed curves on $S_g$ if $i(c_i,c_{i+1})=1$ for $i \in \{1, 2, \ldots, k-1 \}$ and $i(c_i,c_j)=0$ for $|i-j|\ge 2$. For a simple closed curve $c$, denote the \emph{Dehn twist} about $c$ by $T_c$. For a chain of $n-1$ curves $\{c_1,...,c_{n-1}\}$, there exists a \emph{standard homomorphism}
\[
\rho_s: B_n\to \PMod(S_{g,p})
\]
such that $\rho_s(\sigma_i)=T_{c_i}$. There also exists a \emph{negative-standard homomorphism}
\[
\rho_{-s}: B_n\to \PMod(S_{g,p}),
\]
where $\rho_s(\sigma_i)=T_{c_i}^{-1}$. 
The standard homomorphism $\rho_s$ is induced by the hyper-elliptic branched cover of $S_g$ over the sphere, which is ``some manipulation of surfaces" as Mirzakhani conjectured.

We now describe a construction which turns a homomorphism into a new one, called a \emph{transvection}. Let $\rho: B_n\to \PMod(S_{g,p})$ be a homomorphism and $\phi \in \PMod(S_{g,p})$ be an element that commutes with every element of $\rho(B_n)$. We define a new homomorphism
$$\rho_\phi:B_n\to \PMod(S_{g,p})$$
 by $\rho_\phi(\sigma_i)=\phi\rho(\sigma_i)$. This new homomorphism $\rho_\phi$ is called a \emph{transvection} of $\rho$. For example, if the image of $B_n$ under $\rho$ is a cyclic group, then $\rho$ is a transvection of the trivial homomorphism. In this paper, we prove the following result.
\begin{theorem}\label{main1}
For $k\ge 13$, $n=2k$ or $n=2k+1$, and $g\le 2k-3$, any homomorphism $\rho:B_n\to \PMod(S_{g,p})$ is either trivial, $\rho_s$, or $\rho_{-s}$ up to transvection. \end{theorem}

\p{Three remarks about Theorem \ref{main1}.}
\begin{itemize}\item
Theorem \ref{main1} is sharp because other homomorphisms appear when $g=n-2$. This can be seen through the degree $d$ \emph{super-hyperelliptic} family 
\[
y^d=(x-x_1)...(x-x_n),
\]
which is an $S_g$-bundle over $\Conf_n(\C)$ for some $g$, where the monodromy is denoted as $\rho_d$. The representation determined by the degree $d$ super-hyperelliptic family has been computed by \cite{McMullen} and \cite{GHASWALA}. By the Hurwitz formula, $g=(d-1)(k-1)$ for $n=2k$. When $d=2$, the representation is exactly $\rho_s$. For $d=3$ and $n=2k$, the genus $g=2(k-1)=2k-2=n-2$. It is not hard to show that $\rho_3$ is not a transvection of $\rho_s, \rho_{-s}$, or the trivial homomorphism.
\item
If we replace $\PMod(S_{g,p})$ with the mapping class group $\Mod(S_{g,p})$ fixing punctures as a set, the result is false. This is because there is an embedding of the braid group $B_n$ into the surface braid group $B_n(S_g)$, induced by an embedding $\mathbb{C}\hookrightarrow S_g$. As $B_n(S_g)$ in turn embeds into $\Mod(S_{g,n})$, the resulting composite homomorphism is a counterexample to our main theorem.
\item
Unlike \cite{Castel}, we use punctures instead of boundary components for the definition of $\PMod(S_{g,p})$ for simplicity. If we replace punctures with boundary components as well, our argument can provide the same results.
\end{itemize}
Aramayona--Souto \cite{AS} has classified homomorphisms $\PMod(X)\to \PMod(Y)$ for two surfaces $X,Y$ such that either $g(Y)<2g(X)-1$, or $g(Y)=2g(X)-1$ when $Y$ is not closed. Our result partially recovers their theorem in a reduced range but with a slight improvement, as our result additionally holds when $Y$ is a closed surface and $g(Y)=2g(X)-1$.
\begin{corollary}\label{maincor}
Let $g\ge 23$ and $h<2g$. Then any homomorphism $\PMod(S_{g,p})\to \PMod(S_{h,q})$ is either trivial or induced by an embedding.
\end{corollary}

\p{Prior results.} 
There are many results showing that homomorphisms between mapping class groups are induced by inclusions of surfaces; see the work of Ivanov \cite{ivanov}, Irmak \cite{irmak} \cite{irmak2}, Shackleton \cite{shack}, Bell--Margalit \cite{BM}, Castel \cite{Castel}, and Aramayona--Souto \cite{AS}. Aramayona--Leininger--Souto gives examples of injective homomorphisms between mapping class groups of closed surfaces using a covering construction \cite{als}. In Chen--Kordek--Margalit, the authors show that homomorphisms between braid groups are induced by ``cabling" and ``diagonal" constructions, which are two kinds of geometric manipulations of $2$-disks \cite{CKM}. Moreover, Szepietowski \cite{Szepietowski} and B\"{o}digheimer--Tillmann \cite{BT} construct embeddings of braid groups into mapping class groups that do not send the standard generators of the braid group to Dehn twists.

\p{Comparison of methods.}
Similar to the result of Aramayona—Souto \cite{AS}, the authors also try to prove that the image of a twist is a single Dehn twist. To accomplish this task, Aramayona—Souto uses a result of Bridson showing that the image of a power of Dehn twist is a multitwist \cite{bridson}. However, the analog of Bridson's theorem for $B_n$, replacing Dehn twists with half-twists, is not true as $B_n$ has nontrivial abelianization. Instead, we start with the totally symmetric set idea, as discussed in the next remark. Another similarity is that  both \cite{AS} and the current paper use results about homomorphisms from either braid groups or mapping class groups into symmetric groups. One challenge of the current paper is addressing the abelianization of $B_n$, due to which we need to classify homomorphisms up to transvection. Aramayona--Souto does not need to consider this because mapping class groups are perfect.

The argument of Chen--Kordek--Margalit \cite{CKM} does not apply to this problem because it takes advantage of the fact that roots of the central elements of $B_n$, which are analogues of finite order elements in mapping class groups, are classified. However, finite order elements in $\PMod(S_{g,p})$ have not yet been fully classified. The inductive argument in Chen--Kordek--Margalit also fails to apply because we study homomorphisms between two different groups.

\p{Key ingredients of our proof.}
The main tools we use in this paper are canonical reduction systems and \emph{totally symmetric sets}. The notion of totally symmetric sets is useful in studying homomorphisms between groups in general. This idea first appears in \cite{tss} and already has many applications including papers \cite{CKLP}, \cite{KordekCaplinger} and \cite{LiPartin}. We will give an introduction of totally symmetric sets in Section 2. 

Let $\rho:B_n\to \PMod(S_{g,p})$ be a homomorphism. We use totally symmetric sets to constrain the geometry of the canonical reduction system of $\rho(X)$, where $X$ is a totally symmetric subset of $B_n$. Then, Thurston's theory can further constrain the exact image $\rho(X)$.

Another key ingredient of our proof are the rigidity results of Lin \cite{lin} and Artin \cite{artin} concerning homomorphisms from $B_n$ to the symmetric group $\Sigma_m$, which are stated in Section 7. These rigidity results are used in Sections 8, 9, and 10 to determine how $\rho(B_n)$ acts on canonical reduction systems and can help determine $\rho(B_n)$ exactly.

\p{Genus bound on the target mapping class group.}
When the target has higher genus, this type of argument fails catastrophically. As we have remarked after Theorem 1.1, super-hyperelliptic families form one counterexample in the higher genus case. For homomorphisms between mapping class groups, strange covering space constructions arise sending pseudo-Anosov elements to Dehn multitwists; e.g., see \cite{als}. If the target has large genus, the totally symmetric sets of its mapping class group can be very complicated. Another difficulty in the large genus case is that the mapping class group then contains large finite subgroups, which allows more possibilities for classification, further complicating the analysis. However, we still expect that all injective homomorphisms between mapping class groups come from geometric constructions such as covering space constructions.

\p{Outline of the paper.} Section 2 introduces totally symmetric sets and totally symmetric multicurves, the latter of which are then classified in Section 3. Section 4 sketches the proof of Theorem \ref{main1} and breaks the proof into five cases, which are addressed in Sections 5, 6, 8, 9 and 10. Section 7 states the rigidity theorems of Lin and Artin, which will be used in later sections. Lastly, Corollary \ref{maincor} is proven in Section 11. 

\p{Acknowledgments.} We acknowledge the financial support of the National Science Foundation via Grant Nos. DMS - 2005409 and the Caltech SURF program through Prof. Vladimir Markovic's McArthur chair account and the Mark Reinecke SURF Fellowship.  This paper started as a joint paper of Lei Chen with Kevin Kordek and Dan Margalit, and transformed into a SURF program project at Caltech. The authors would like to thank Kevin and Dan for all the helpful conversations regarding this paper. The authors would also like to thank the anonymous referee for giving very useful advice to the current paper.

\section{Totally symmetric sets}
In this section, we introduce the main technical tools of the paper, namely totally symmetric sets and totally symmetric sets.  We note that most of the material in this chapter has already appeared in \cite{tss}, but  we include them for the sake of completeness. After giving some examples, we derive some basic properties of totally symmetric sets and develop their relationship with canonical reduction systems.

\subsection{Totally symmetric subsets of groups}

Let $X$ be an ordered subset of a group $G$.  

\begin{definition} $X$ is a \emph{totally symmetric subset} of $G$ if
\begin{itemize}
\item the elements of $X$ commute pairwise, and
\item every permutation of $X$ can be achieved via conjugation by an element of $G$.
\end{itemize}
\end{definition}

\noindent If $X =\{x_1,\dots,x_m\}$ is a totally symmetric subset of $G$, we can form new totally symmetric subsets from $X$. In particular,
$$ X^k = \{ x_1^k ,\dots, x_m^k \} $$
and 
$$ X' = \{x_1x_2^{-1},x_1x_3^{-1},\dots,x_1x_m^{-1}\} $$
are both totally symmetric sets.

\p{The symmetric set $X_n$.} 
A key example is the set of odd-indexed generators of $B_n$
$$ X_n = \{\sigma_1,\sigma_3,\sigma_5,\dots,\sigma_{\ell}\}, $$
where $\ell$ is the largest odd integer less than $n$.

As above, $X_n'$ is also totally symmetric in $B_n$. For $n\ge 6$, any permutation of the elements of $X_n$ can be achieved by conjugating by an element of the commutator subgroup $B_n'$. For example, given any element $g\in B_n$ that transposes $\sigma_1$ and $\sigma_3$ and fixes every other element of $X_n$ by conjugation, the action of the element $g\sigma_5^k$ is the same as the action of $g$ on $X_n$. Thus we can choose an appropriate $k$ so that $g\sigma_5^k\in B_n'$. 

We conclude the above discussion as the following fact.
\begin{fact}\label{braidsymmetricset}
The set $X_n$ is a totally symmetric subset of $B_n$. Likewise, $X_n'$ is a totally symmetric subset of $B_n'$. In particular, every permutation of $X_n$ can be achieved by elements in $B_n'$.
\end{fact}

A basic property of totally symmetric sets is the following lemma, which is proven in \cite[Lemma 2.1]{tss}.
\begin{lemma}
\label{lem:tss}
Let $X$ be a totally symmetric subset of a group $G$ and let $\rho : G \to H$ be a homomorphism of groups.  Then $\rho(X)$ is either a singleton or a totally symmetric set of cardinality $|X|$.
\end{lemma}

\subsection{From totally symmetric sets to totally symmetric multicurves}

The topological version of totally symmetric sets are \emph{totally symmetric multicurves}. 

\begin{definition}
Let $X$ be a set. An \emph{$X$-labeled} multicurve $M$ is a multicurve where each component of $M$ is associated with a non-empty subset of $X$. This subset is called its label.\end{definition}
The labelling can also be viewed as a map from the set of components of $M$ to the power set of $X$.

There is a canonical action of the symmetric group $\Sigma_{|X|}$ on the set of $X$-labeled multicurves, where $\Sigma_{|X|}$ acts on the labels of each component. The mapping class group $\Mod(S)$ also acts on the set of $X$-labeled multicurves by acting on the set of multicurves.

\begin{definition} An $X$-labeled multicurve $M$ on a surface $S$ is \emph{totally symmetric} if for every $\sigma \in \Sigma_{|X|}$, there is an $f_\sigma \in \Mod(S)$ so that $\sigma \cdot M = f_\sigma \cdot M$. \end{definition}

An $X$-labeled multicurve has the \emph{trivial label} if each curve in the multicurve has the entire set $X$ as its label. Note that every multicurve with the trivial label is totally symmetric.

\p{Canonical reduction systems.} Nielsen and Thurston classified elements of $\Mod(S)$ into three types: periodic, reducible, and pseudo-Anosov (see, e.g., \cite[Chapter 13]{FM}). Every element $f$ of $\Mod(S)$ has an associated multicurve $\CRS(f)$, which is called the \emph{canonical reduction system}. Up to taking a power, $f$ is homotopic to either the identity or a pseudo-Anosov map on the complement of $\CRS(f)$. Canonical reduction systems can be considered as the analog of the Jordan decomposition of a matrix for the mapping class group. We list some facts about canonical reduction systems that are used in this paper (for more background, see Birman--Lubotzky--McCarthy \cite{blm}).
\noindent
\begin{lemma}For $f,g \in \Mod(S)$, we have the following.
\begin{itemize}\item[(a)] $\CRS(f) = \emptyset$ if and only if $f$ is periodic or pseudo-Anosov.
\item[(b)] If $f$ and $g$ commute, then for any $c\in \CRS(f)$ and $d\in \CRS(g)$, we have $i(c,d)=0$.
\item[(c)] For any $f$ and $g$ we have $\CRS(gfg^{-1}) = g(\CRS(f))$.
\end{itemize}
\end{lemma}

From any totally symmetric subset $X$ of $\Mod(S)$, we can construct the $X$-labeled multicurve $\CRS(X)$ as follows: The multicurve $\CRS(X)$ is the union of $\CRS(x)$ for all $x\in X$. The label of each component $c \in \CRS(X)$ is $\{ x \in X | c \in \CRS(x) \}\subset X$. We then have the following fact.

\begin{fact}
\label{lem:symcrs}
If $X$ is a totally symmetric subset of $\Mod(S)$ then $\CRS(X)$ is a totally symmetric multicurve.
\end{fact}
Fact \ref{lem:symcrs} allows us to determine the image $\rho(X_n)$ through constraining totally symmetric multicurves.    
The totally symmetric multicurves $\CRS(X_n)$ and $\CRS(X_n')$ for $X_n$ and $X_n'$ in $B_n$ are indicated in Figure~\ref{fig:symcrs}. 
\medskip
\begin{figure}[H]
\centering
\begin{subfigure}{.49\textwidth}
  \centering
  \includegraphics[width=0.85\linewidth]{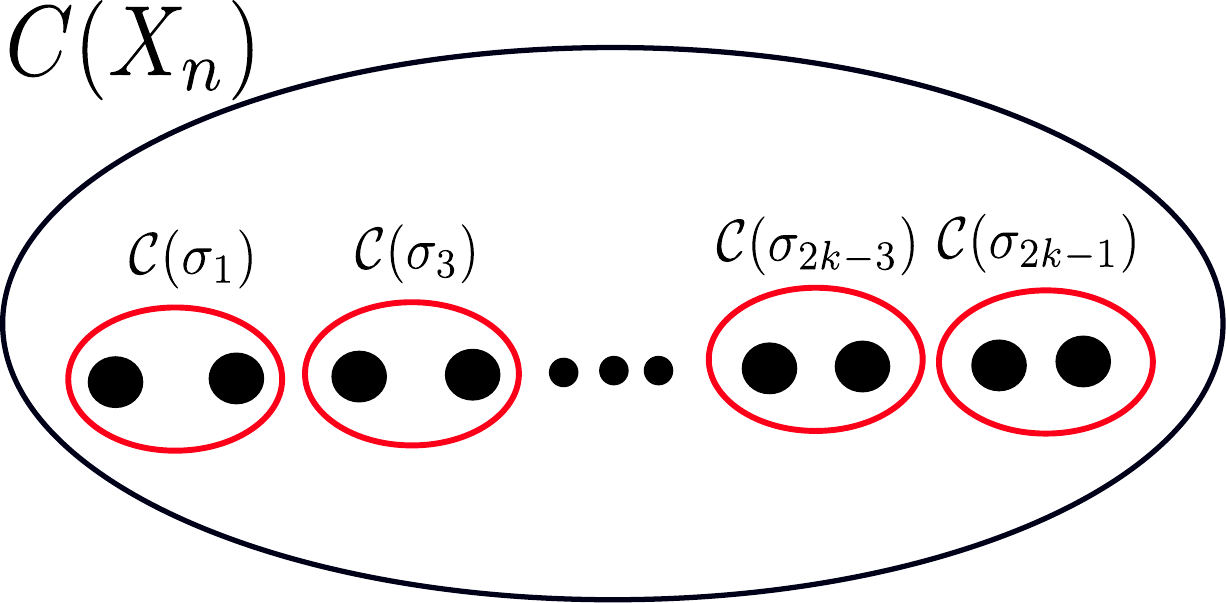}
  \caption{}
\end{subfigure}
\begin{subfigure}{.49\textwidth}
  \centering
  \includegraphics[width=0.85\linewidth]{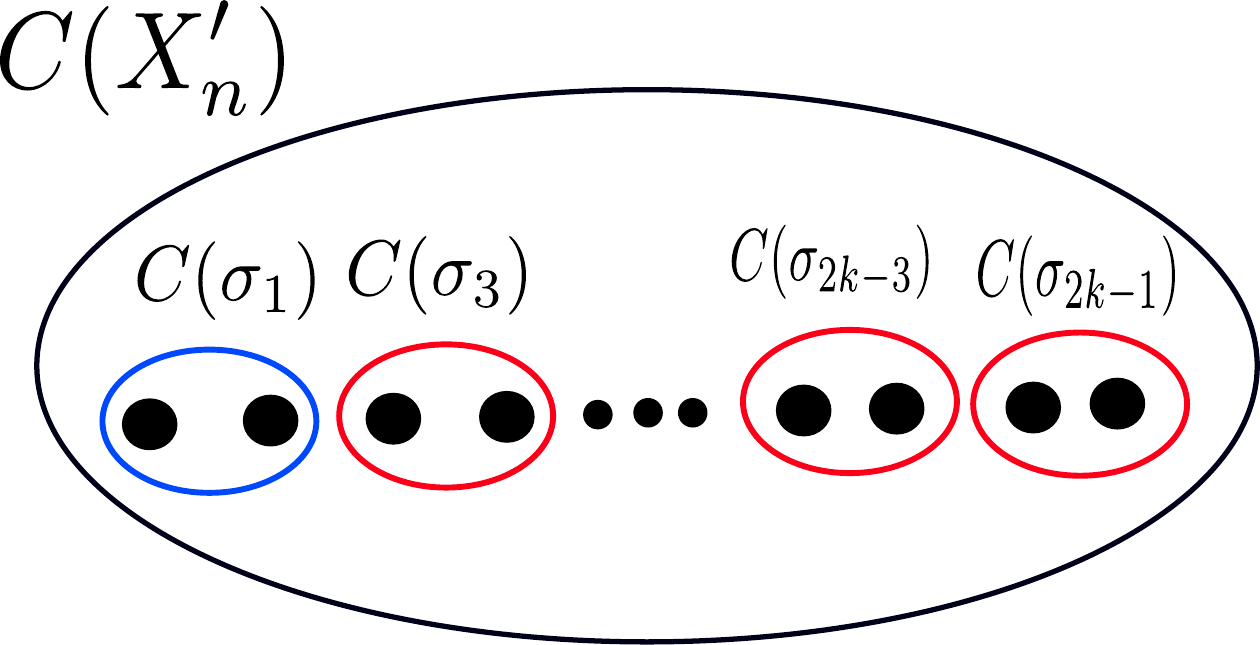}
  \caption{}
\end{subfigure}

\caption{\newline (A) Each element of $X_n$ has exactly one component of $\CRS(X_n)$ as its canonical reduction system, as labelled. \newline (B) For each element $\sigma_1\sigma_i^{-1}$ of $X_n'$, $\CRS(\sigma_1\sigma_i^{-1})=\CRS(\sigma_1) \cup \CRS(\sigma_i)$ for $1<i<n$ odd.}
\label{fig:symcrs}
\end{figure}

\section{Classifying totally symmetric multicurves}

In this section, we classify totally symmetric multicurves on $S_{g,p}$ in order to constrain the totally symmetric sets in $\Mod(S_{g,p})$. Let $M$ be a totally symmetric $X$-labeled multicurve in $S_{g,p}$ for a finite set $X$ with $|X|=k$.  We divide the components of $M$ into three types as follows.  Let $c$ be a component of $M$ and let $l(c)$ be  its label.  We say that $c$ is:
\begin{itemize}
\item of type A if $|l(c)|=1$,
\item of type I if $|l(c)| = k-1$,
\item of type C  if $|l(c)| = k$.
\end{itemize}

The main goal of this section is to prove the following proposition, which says that in some range of $g$ and $k$, every component of a totally symmetric multicurve is of type A, I or C.
\begin{proposition}
\label{prop:abc}
If $M$ is a totally symmetric $X$-labeled multicurve in $S_g$ with $|X|=k$ and $k^2-k > 6g+6$, then every component of $M$ is of type A, I, or C. In particular, this statement holds when $n \geq 26$, $g \leq n-3$, and $k = \floor{n/2}$. 
\end{proposition}

\begin{proof}

If $k \leq 3$, then the proposition is vacuous. Let $c$ be a component of $M$ with $|l(c)| =\ell$.  Since $M$ is totally symmetric, for each subset $A\subset X$ of cardinality $\ell$, there is a component of $M$ whose label is $A$.  In particular, $|M| \geq {k \choose \ell}$.  However, on $S_g$, the maximal number of disjoint curves are $3g+3$. Therefore, we have that 
\[
{k \choose \ell} \leq 3g+3.
\]
Assume that $c$ is not of type A, I, or C.  This means that $2 \leq \ell \leq k-2$.  Since $k \geq 4$ and $\ell \notin \{0,1,k-1,k\}$, we have that ${k \choose \ell} \geq {k \choose 2}$.  We thus conclude that ${k \choose 2} \leq 3g+3$.  In other words, $k(k-1) \leq 6g+6$, which contradicts the assumption. Thus every component of $M$ is of type A, I, or C.

Note that if  $n \geq 26$ then
\[
\floor{n/2}\left(\floor{n/2}-1\right) \geq \frac{n-1}{2}\left(12 \right) = 6n-6 \ge 6(g+3)-6=6g+12 >6g+6,
\]
which gives us the range as in the assumption of the statement.
\end{proof}

%%%
%%%
%%%

\section{Outline of the proof}
\label{sec:outline}

In this section we present an outline of the remainder of the paper, which is devoted to the proof of Theorem~\ref{main1}.  Given a homomorphism 
\[
\rho : B_n \to \PMod(S_{g,p}),\] let 
\[
\rho_0: B_n\to \Mod(S_g)
\]
be the composition of $\rho$ with the natural projection $\PMod(S_{g,p})\to \Mod(S_g)$. We first work with $\rho_0$ because many statements, such as Proposition \ref{prop:abc}, only work with closed surfaces. We then use our results for $\rho_0$ to study $\rho$.

To classify $\rho_0$, we consider the totally symmetric subset $X_n\subset B_n$. The image of $X_n$ is a totally symmetric set $\rho_0(X_n)$ in $\Mod(S_g)$ with cardinality either 1 or $\floor{n/2}$ (Lemma~\ref{lem:tss}). If $\rho_0(X_n)$ has cardinality $1$, then since $B_n'$ is normally generated by $\sigma_1\sigma_3^{-1}$ \cite[Remark 1.10]{lin}, it follows that $\rho$ is cyclic. 

It thus remains to consider the case where $\rho_0(X_n)$ is a totally symmetric set of cardinality $\floor{n/2}$.   As in Fact \ref{lem:symcrs}, the $X_n$-labeled multicurve $\CRS(\rho_0(X_n))$ is totally symmetric.  By Proposition~\ref{prop:abc}, each curve of $\CRS(\rho_0(X_n))$ is of type A, I, or C. We break the proof into the following five cases.
\begin{enumerate}
\item $\CRS(\rho_0(X_n))$ is empty and $\rho_0(\sigma_1)$ is periodic, 
\item $\CRS(\rho_0(X_n))$ is empty and $\rho_0(\sigma_1)$ is pseudo-Anosov,
\item $\CRS(\rho_0(X_n))$ contains only curves of type $C$,
\item $\CRS(\rho_0(X_n))$ contains curves of type I, and
\item $\CRS(\rho_0(X_n))$ contains curves of type A but no curves of type $I$ .
\end{enumerate}
In the first three cases, we show that $\rho$ is cyclic.  The three cases are detailed in Corollary~\ref{cor:torsion}, Proposition~\ref{prop:pA}, and Proposition~\ref{prop:c}, proved in Sections~\ref{sec:torsion}, \ref{sec:pA}, and~\ref{sec:c}, respectively. In the fourth case, we show that $\CRS(\rho_0(X_n))$ cannot contain curves of type I.  This is the content of Proposition~\ref{prop:b} in Section~\ref{sec:b}.  

After proving the first four cases, we complete the proof of Theorem~\ref{main1} in Section~\ref{sec:pf}.   The goal is to show that there are exactly $\floor{n/2}$ curves of type A, exactly one for each $\rho(\sigma_i)$, and further that each of these curves is a nonseparating curve in $S_g$.  We then use this to conclude that (modulo transvections) each $\rho(\sigma_i)$ is a power of the corresponding Dehn twist.  We finally apply the braid relation to show that the power is $\pm 1$, and the theorem follows from there. 

In Section 11, we end this paper by proving Corollary \ref{maincor}.

%%%
%%%
%%%

\section{Case (1): $\rho_0(\sigma_i)$ is a torsion element}
\label{sec:torsion}

This section addresses Case (1) of Theorem \ref{main1}, as described in Section \ref{sec:outline}, by proving Corollary \ref{cor:torsion}. We start with the following lemma connecting $\rho_0$ and $\rho$, which will be used in the next few sections. Recall that $\rho_0:B_n \to \Mod(S_g)$ is the composition of $\rho:B_n \to \PMod(S_{g,p})$ with the forgetful map $\PMod(S_{g,p}) \to \Mod(S_g)$.

\begin{lemma}
\label{lem:rho0torho}
If $\rho_0$ is cyclic, then $\rho$ is cyclic.
\end{lemma}

\begin{proof}
When restricted to $B_n'$, the homomorphism $\rho|_{B_n'}: B_n'\to \PMod(S_{g,p})$ has image in the kernel of $\PMod(S_{g,p})\to \Mod(S_g)$. The kernel is the point-pushing subgroup $PB_p(S_g)$ as in \cite[Theorem 9.1]{FM}. It is a classical fact that $PB_p(S_g)$ is locally-indicable; i.e., any nontrivial finitely generated subgroup surjects onto $\mathbb{Z}$. However, $B_n'$ is a finitely generated perfect group for $n\ge 5$ by Gorin--Lin \cite{GorinLin}, which cannot be mapped nontrivially into a locally-indicable group.
\end{proof}

The following proposition bounds the genus of a surface whose mapping class group contains a totally symmetric set of finite order elements.

\begin{proposition}
\label{prop:mcgtorsion}
If $\PMod(S_{g,p})$ contains a totally symmetric set consisting of $k$ periodic elements, then $g \geq 2^{k-3}-1$.
\end{proposition}

To prove this statement, we use the following result due to Chen, Kordek, and Margalit, which appears in \cite{CKLP}.

\begin{lemma}
\label{lem:torsion}
Let $G$ be a group and let $X \subseteq G$ be a totally symmetric subset with $|X| = k$.  Suppose that each element of $X$ has finite order.  Then the group generated by $X$  is a finite group whose order is greater than or equal to $2^{k-1}$.
\end{lemma}
\begin{proof}
See \cite[Lemma 2.3]{CKLP}.
\end{proof}

We can now proceed with the proof of Proposition \ref{prop:mcgtorsion}.

\begin{proof}[Proof of Proposition~\ref{prop:mcgtorsion}]

Let $X$ be a totally symmetric subset of $\PMod(S_{g,p})$ with $|X|=k$.  By Lemma~\ref{lem:torsion}, the group $\PMod(S_{g,p})$ contains an abelian subgroup of order at least $2^{k-1}$.  On the other hand, the order of a finite abelian subgroup of $\PMod(S_{g,p})$ is bounded above by $4g+4$ \cite{ab1,ab2}.  We thus have $2^{k-1} \leq 4g+4$, and the proposition follows.
\end{proof}

We now prove the main result of this section.

\begin{corollary}
\label{cor:torsion}
Suppose that $g < 2^{k-3}-1$, where $k = \floor{n/2}$. Let $\rho: B_n \to \PMod(S_{g,p})$ be a homomorphism and $\rho_0:B_n \to \Mod(S_g)$ be the composition of $\rho$ with the forgetful map $\PMod(S_{g,p}) \to \Mod(S_g)$. If $\rho_0(\sigma_1)$ is periodic, then $\rho$ is cyclic. In particular, $n \geq 14$ and $g \leq n-3$ satisfies the assumptions on $n$ and $g$.
\end{corollary}

\begin{proof}

Since $\rho_0(\sigma_1)$ is periodic, it follows that $\rho_0(X_n)$ is a totally symmetric subset of $\Mod(S_{g,p})$ consisting of periodic elements.  According to Lemma \ref{lem:tss}, the set $\rho_0(X_n)$ is either a singleton or contains exactly $\floor{n/2}$ elements.  We treat the two cases in turn.

If $\rho_0(X_n)$ is a singleton, then $\rho_0(\sigma_1)=\rho_0(\sigma_3)$, so $\rho_0(\sigma_1\sigma_3^{-1}) = 1$, and all conjugates of $\sigma_1\sigma_3^{-1}$ lie in the kernel of $\rho_0$. As $\sigma_1\sigma_3^{-1}$ normally generates the commutator subgroup $B_n'$, then $B_n' \subseteq \ker \rho_0$, so $\rho_0$ is cyclic. By Lemma \ref{lem:rho0torho}, we know that $\rho$ is cyclic, as desired.

If $\rho_0(X_n)$ is not a singleton, then it has $\floor{n/2}$ elements, so the assumption on $g$ and $n$ contradicts Proposition \ref{prop:mcgtorsion}.

Lastly, when $n\geq 14$, we verify that $g \leq n-3 < 2^{\floor{n/2} - 3} - 1$, as desired.
\end{proof}

The next corollary is proved similarly and will be used in later sections.

\begin{corollary}
\label{cor:torsion'}
Suppose that $g < 2^{k-3}-1$, where $k = \floor{n/2}-1$. Let $\gamma: B_n \to \PMod(S_{g,p})$ be a homomorphism and $\gamma_0:B_n \to \Mod(S_g)$ be the composition of $\gamma$ with the forgetful map $\PMod(S_{g,p}) \to \Mod(S_g)$. If $\gamma_0(\sigma_1\sigma_3^{-1})$ is periodic, then $\gamma$ is cyclic. In particular, $n \geq 16$ and $g \leq n-3$ satisfies the assumptions on $n$ and $g$.
\end{corollary}

\begin{proof}
Again, by Lemma \ref{lem:tss}, $\gamma_0(X_n')$ is a singleton or consists of $\floor{n/2} - 1$ elements. If $\gamma_0(X_n')$ is a singleton, then as $\gamma_0(\sigma_1\sigma_3^{-1}) = \gamma_0(\sigma_1\sigma_5^{-1})$, we know \[
\gamma_0(\sigma_3\sigma_5^{-1}) = \gamma_0(\sigma_1\sigma_3^{-1})^{-1} \gamma_0(\sigma_1\sigma_5^{-1}) = 1.\] As $\sigma_3\sigma_5^{-1}$ and $\sigma_1\sigma_3^{-1}$ are conjugate, then $\gamma_0(\sigma_1\sigma_3^{-1}) = 1$, so $\gamma_0(B_n') = 1$ implying $\gamma_0$ is cyclic. By Lemma \ref{lem:rho0torho}, $\gamma$ is cyclic.

In the second case, $\gamma_0$ is trivial by Proposition \ref{prop:mcgtorsion}, as in Corollary $\ref{cor:torsion}$. Again, by Lemma \ref{lem:rho0torho}, $\gamma$ is cyclic.
\end{proof}

\section{Case (2): $\rho_0(\sigma_i)$ is pseudo-Anosov}
\label{sec:pA}

The goal of this section is to prove the following proposition, which completely addresses Case (2) of the proof of Theorem~\ref{main1}, as outlined in Section~\ref{sec:outline}.

\begin{proposition}
\label{prop:pA}
Let $n \geq 16$ and  $g \leq n-3$. Let $\rho: B_n\rightarrow \PMod(S_{g,p})$.  If $\rho_0(\sigma_1)$ is pseudo-Anosov, then $\rho$ is cyclic.
\end{proposition}

The key fact in the proof, that the centralizer of a pseudo-Anosov mapping class is virtually cyclic, is due to McCarthy \cite{mccarthy}.

\begin{proof}[Proof of Proposition~\ref{prop:pA}]

As $\sigma_1$ and $\sigma_3$ commute, then $\rho_0(\sigma_3)$ is in the centralizer of $\rho_0(\sigma_1)$.

As the centralizer of a pseudo-Anosov element is virtually cyclic and the pseudo-Anosov element generates a cyclic subgroup in its centralizer, $\rho_0(\sigma_1)^{k_1} = \rho_0(\sigma_3)^{k_2}$ for two integers $k_1,k_2$ where $k_2\neq 0$. By the conjugation of pairs $(\sigma_1,\sigma_3)$ and $(\sigma_3,\sigma_1)$, we also obtain $\rho_0(\sigma_3)^{k_1} = \rho_0(\sigma_1)^{k_2}$. Thus, we know that $\rho_0(\sigma_3)^{k_1+k_2} = \rho_0(\sigma_1)^{k_2+k_1}$ by multiplying the two equations $\rho_0(\sigma_1)^{k_1} = \rho_0(\sigma_3)^{k_2}$ and $\rho_0(\sigma_3)^{k_1} = \rho_0(\sigma_1)^{k_2}$. 

If $k_1+k_2\neq 0$, then we know that $\rho_0(\sigma_1\sigma_3^{-1})$ is periodic, which implies that $\rho$ is cyclic by by Corollary \ref{cor:torsion'}. If $k_1+k_2=0$, then we know that 
 $\rho_0(\sigma_1\sigma_3)$ is periodic, then by conjugation, we know that  $\rho_0(\sigma_3\sigma_5)$ is periodic. This implies that $\rho_0(\sigma_1\sigma_5^{-1})$ is periodic by multiplying $\rho_0(\sigma_1\sigma_3)$ with the inverse of $\rho_0(\sigma_3\sigma_5)$. Therefore $\rho_0(\sigma_1\sigma_3^{-1})$ is periodic by conjugation, which implies that $\rho$ is cyclic by by Corollary \ref{cor:torsion'}

\end{proof}

%%%
%%%
%%%

\section{Lin's dichotomies and Artin's theorem}

In this section, we introduce results of Lin and Artin about homomorphisms from $B_n$ to the symmetric group $\Sigma_m$, which will be used to classify actions of $B_n$ on multicurves.  

\begin{theorem}[Theorem A of Lin \cite{lin}] 
\label{lin1}
For $6<n<m<2n$, every transitive homomorphism $B_n\to \Sigma_m$ is cyclic.\end{theorem}
\begin{theorem}[Theorem F of Lin \cite{lin}]\label{lin2}
For $n>4$ and $m<n$, every homomorphism from the commutator subgroup $B_n'\to \Sigma_m$ is trivial.
\end{theorem}
The natural action of $B_n$ on punctures induces a homomorphism $\pi: B_n\to \Sigma_n$, which is called the \emph{standard representation}. Artin proved the following result about rigidity of homomorphisms $B_n\to \Sigma_n$ \cite{artin}.
\begin{theorem}[Artin's theorem]\label{artin}
For $n>4$, every transitive homomorphism  $B_n\to \Sigma_n$ is either cyclic or standard.
\end{theorem}

\section{Case (3): $\CRS(\rho_0(X_n))$ only contains curves of type C}
\label{sec:c}

The goal of this section is to prove the following proposition, which completely addresses Case (3) of the proof of Theorem~\ref{main1} as outlined in Section~\ref{sec:outline}. Recall that $\rho:B_n\to \PMod(S_{g,p})$ is a homomorphism and $\rho_0: B_n\to \Mod(S_g)$ is the composition of $\rho_0$ with the natural projection $\PMod(S_{g,p})\to \Mod(S_g)$.

\begin{proposition}
\label{prop:c}
Let $n \geq 26$ and $g \leq n-3$.  Assume that the labeled multicurve $\CRS(\rho_0(X_n))$ contains only curves of type C.  Then $\rho$ is cyclic.
\end{proposition}

In order to prove this proposition, we  use the following lemma, which will be used  in both Section 8 and Section 9. Notice that the following lemma applies only to homomorphisms $B_n\to \Mod(S_g)$.

\begin{lemma}
\label{lem:curveaction}
Let $n > 6$ and let $g \leq n-3$.  Let $\gamma : B_n \to \Mod(S_g)$ be a homomorphism.  Suppose that $\gamma(B_n)$ preserves the isotopy class of a multicurve $M$.  Then
\begin{enumerate}
\item the induced action of $B_n$ on the set of components of $S_g \setminus M$ is cyclic, and
\item the action of $B_n$ on $M$ is cyclic. 
\end{enumerate} 
\end{lemma}

\begin{proof} The multicurve $M$ divides $S_g$ into compact subsurfaces $R_1,\dots,R_r$ so that \[
\bigcup_i R_i = S_g \setminus M.\]  Each $R_i$ has negative Euler characteristic, and so $r \leq 2g-2$.  Further, since $g \leq n-3$, we have $r \leq 2n-8 < 2n$. Since $\gamma(B_n)$ preserves $M$, we know that 
 $\gamma(B_n)$ acts on $\{R_i\}$ by permutation. We will denote this action by $\tau$.
 
By combining Theorem \ref{lin1} and Theorem \ref{artin}, if $B_n$ acts on a set with less than $2n$ elements then the action is either (a) or (b) of the following.
\begin{enumerate}
\item[(a)] $\tau$ has cyclic image;
\item[(b)] There is a subset consisting of $n$ elements on which $B_n$ acts standardly (meaning that there is an identification of the subset with $\{1,\dots,n\}$ so that the action of $B_n$ on the subset is given by the standard homomorphism $B_n \to \Sigma_n$).
\end{enumerate}

\medskip\noindent\textit{Proof of statement (1).}

For the first statement, we need to prove that (b) of the above cannot occur. If it occurs, then $k\geq n$.  Without loss of generality assume that $B_n$ acts standardly on $R_1,\dots,R_n$.  Then $\{R_1,R_2,...,R_n\}$ must be pairwise homeomorphic, and so in particular they must have the same Euler characteristic.  

We first claim that $R_1$, hence each $R_i$ for $1\le i\le n$, is a pair of pants and the boundary of $R_1$ is the set of three distinct curves. Firstly, we have $\chi(S_g) \leq n\chi(R_1)$.  Using the assumption $g \leq n-3$, we find
\[
-2n+8 = 2-2(n-3) \leq 2-2g \leq n\chi(R_1). 
\]
It follows that $8/n-2 \leq \chi(R_1)$, and in particular $-2 < \chi(R_1)$.  Since we also have $\chi(R_1) < 0$, and so $\chi(R_1) = -1$. The only possibility for $R_1$ other than a pair of pants is a torus with one boundary component. However, if $S_g$ contains $n$ tori with one boundary component, the genus $g$ will satisfy $g \geq n$, which contradicts the assumption that $g\le n-3$. If two of the boundary components of $R_1$ are the same curve on $S_g$, then the other curve in the boundary is a separating curve bounding a genus $1$ subsurface. This also contradicts the fact that $g\le n-3$.

We next claim that there is a pair of distinct elements of $\{R_1,\dots,R_n\}$ with nonempty intersection.  We decompose $S_g$ into pairs of pants $R_1,\dots,R_n,Q_1,\dots,Q_m$.  We have $n+m = 2g-2$.  Since $g \leq n-3$ we have
\[
m = 2g-2-n \leq n-8 < n.
\]
Since $n>m$, we know that there exist two elements of $R_1,\dots,R_n$ which share a boundary component. The claim now follows.

By the previous claim, we renumber  $\{R_i\}$ so that $R_1$ and $R_2$ share a boundary component.  Since $R_1$ has only three boundary components, and since $n > 4$, there must be some $R_i$, say $R_3$, that is disjoint from $R_1$.  It follows that there is no element of $\Mod(S_g)$ that transposes $R_2$ and $R_3$ and preserves $R_1$.  In particular, it is impossible for $B_n$ to act standardly on $\{R_1, ...., R_n\}$, as desired, and the first statement is proven.

\bigskip\noindent\textit{Proof of statement (2).}

We now proceed to the second statement.  We will prove the equivalent statement that the action of $B_n'$ on the set of components of $M$ is trivial.  By the first statement, we know that $B_n'$ preserves each $R_i$.  Let $\tau_i$ be the action of $B_n'$ on the connected components of $\partial R_i$. There are two possibilities to consider. The first is that the number of boundary components of $R_i$  is less than $n$ for every $i$. Then, by Theorem \ref{lin2}, we know that $\tau_i$ is trivial. This shows that the action of $B_n'$ on $M$ is trivial.

The second possibility is that the number of boundary components of some $R_i$ is at least $n$.  Assume that $\partial R_1$ has the greatest number of components inside $\{\partial R_i\}$. In this case, we want to show that the whole group $B_n$ acts on $R_1$ and the action on components of $R_1$ is cyclic. This will imply statement (2). When the number of components of $\partial R_1$ is at least $n$,  then $\chi(R_1) \leq 2-n$. Let $R_i^c$ denote the complementary subsurface of $R_1$ in $S_g$, which may be disconnected.  Since $g \leq n-3$
we have
\[
\chi(R_1^c)=(2-2g)-\chi(R_1) \geq 2-2(n-3)-(2-n)=6-n>\chi(R_1).
\]
As $(6-n) - (2-n) > 0$, it follows that $R_1$ cannot be homeomorphic to a subsurface of $R_1^c$.  In particular, $R_1$ is not homeomorphic to any $R_j$ with $j \neq 1$.  Thus, the action of $B_n$ on $\{R_1,\dots,R_r\}$ fixes $R_1$, so the entire group $B_n$, not just $B_n'$, acts on the set of components of  $\partial R_1$.  We apply Theorem \ref{lin1} to the action $\tau_i$.  If $\tau_i$ is standard, we would have a totally symmetric multicurve in $S_g$ with at least $n \geq g+3$ components with label a singleton. However, we have the following claim. 
\begin{claim}
The maximal number of components of totally symmetric multicurves on $S_g$ with a singleton as their labels is $g+1$. 
\end{claim}
\begin{proof}
This can be seen from a simple homology argument. If we have a totally symmetric multicurve $P$ in $S_g$ where $n > g+1$ components have a singleton as their label, then a subset $\{c_1,...,c_l\}$ of $P$ with $l\le g+1$ bound a subsurface. Then as homology classes, we have 
\[
[c_1]+...+[c_l]=0.\]
By transposing $c_1$ with $c_{l+1}$ and fixing every other curve, we obtain
\[
[c_{l+1}]+[c_2]+...+[c_l]=0.\]
Then we have that $[c_1]=[c_{l+1}]$, which implies that $[c_1]=[c_2]=...=[c_{n}]$, meaning $c_1,...,c_n$ are pairwise homologous. However, the maximal number of pairwise homologous disjoint curves on $S_g$ is exactly $g-1$, contradicting our assumption.
\qedhere
\end{proof}
This concludes the proof of Lemma \ref{lem:curveaction} by contradiction.
\end{proof}

\medskip We now prove our main proposition.

\begin{proof}[Proof of Proposition~\ref{prop:c}]
We first prove the corresponding theorem for $\rho_0$. We will split the proof into multiple steps. Our aim is to show that $\rho_0(\sigma_1\sigma_3^{-1})$ has finite order. Then the statement for $\rho$ follows from applying Corollary \ref{cor:torsion'} to the totally symmetric set $X_n'$.  Let $\CRS$ be the multicurve $\CRS(\rho_0(\sigma_1))$, which is equal to $\CRS(\rho_0(\sigma_i))$ for any $i<n$ odd since $\CRS(\rho_0(X_n))$ only contains curves of type C .

\begin{itemize}
\item

\emph{Step 1. $\rho_0$ induces an action of $B_n$ on $\CRS$.}
\\
Each $\sigma_i \in X_n$ preserves $\CRS$.  It remains to check that the $\sigma_i$ for even $i<n$ preserve $\CRS$.  Every element $\sigma_i$ for even $i<n$ commutes with some $\sigma_j$ for odd $j<n$. This implies that $\rho_0(\sigma_i)$ preserves $\CRS(\rho_0(\sigma_j))=\CRS$. Since $B_n$ is generated by the set $\{\sigma_i \ \mid \ 1\le i<n\}$, we know that $\rho_0(B_n)$ preserves $\CRS$.

\medskip \item
\emph{Step 2. Define $\overline{\rho_0}$.}

Consider the surface $S$ obtained from $S_g$ by deleting a representative of $\CRS$.  Denote the connected components of $S$ by $R_1,\dots,R_r$.  It follows from Step 1 that there is a cutting homomorphism
\[
\rho_0' : B_n \to \Mod(R_1 \cup \cdots \cup R_r).
\]
By the first statement of Lemma~\ref{lem:curveaction}, the action of $B_n$ on the set $\{R_i\}$ induced by $\rho_0'$ is cyclic, with kernel denoted as $B_n^K$, which
contains $B_n'$.   The homomorphism
\[
\bar \rho_0 : B_n^K \to \Mod(R_1) \times \cdots \times \Mod(R_r).
\]
is then obtained as a restriction of $\rho_0'$.

\medskip \item 
\emph{Step 3. $\overline{\rho_0}(\sigma_1\sigma_3^{-1})$ has finite order.}

Let $\bar \rho_0^i$ denote the composition of $\bar \rho_0$ with projection to the $i$th factor.  Let $k$ be the smallest number such that $\sigma_i^k\in B_n^K$. To prove step 3, it suffices to show that there is an $\ell$ so that $\bar \rho_0^i(\sigma_1^k)^\ell$ is equal to $\bar \rho_0^i(\sigma_3^k)^\ell$.  

Fix some $i$ in $\{1,\dots,k\}$.  Since $\sigma_1$ and $\sigma_3$ are conjugate by an element of $B_n'$,  it follows that $\bar \rho_0^i(\sigma_1^k)$ is conjugate to $\bar \rho_0^i(\sigma_3^k)$. By definition of the $R_i$, each of $\bar \rho_0^i(\sigma_1^k)$ and $\bar \rho_0^i(\sigma_3^k)$ has empty canonical reduction system.  If they are periodic, then there is an $\ell$ so that $\bar \rho_0^i(\sigma_1^k)^\ell$ and $\bar \rho_0^i(\sigma_3^k)^\ell$ are both trivial.  If they are pseudo-Anosov, then since they are conjugate and commute, there is an $\ell$ so that $\bar \rho_0^i(\sigma_1^k)^\ell$ and $\bar \rho_0^i(\sigma_3^k)^\ell$ are equal (see Section 6 for more explanation).

\medskip \item 
\emph{Step 4. $\rho_0(\sigma_1\sigma_3^{-1})$ has finite order}

Let $q$ be the order of $\overline{\rho_0}(\sigma_1\sigma_3^{-1})$.  Step 3 implies  that $\rho_0((\sigma_1 \sigma_3^{-1})^q)$ is equal to a multitwist about the components of $\CRS$. Since $(\sigma_1 \sigma_3^{-1})^q$ and $(\sigma_1 \sigma_5^{-1})^q$ are conjugate by some element $h$ of $B_n'$, we know that $\rho_0(\sigma_1 \sigma_5^{-1})^q$ is a multitwist about components of $\CRS$ as well. By (2) of Lemma~\ref{lem:curveaction} and the fact that $h\in B_n'$, we know that \[
\rho_0(\sigma_1 \sigma_3^{-1})^q=\rho_0(\sigma_1 \sigma_5^{-1})^q.\]
Therefore, $\rho_0(\sigma_3\sigma_5^{-1})^q=1$. Since $\sigma_1 \sigma_3^{-1}$ is conjugate to $\sigma_3 \sigma_5^{-1}$ in $B_n$, we know that $\rho_0(\sigma_1\sigma_3^{-1})^q=1$ as well. As $\rho_0(\sigma_1\sigma_3^{-1})$ has finite order, then $\rho_0$ of each element of the totally symmetric set $X_n'$ in $B_n$ has finite order.
\end{itemize}

By Corollary \ref{cor:torsion'}, we know that $\rho$ is cyclic.
\end{proof}

%%%
%%%
%%%

\section{Case (4): $\CRS(\rho_0(X_n))$ cannot contain curves of type I}
\label{sec:b}
In this case, we will prove that $B_n$ preserves the union of all type $I$ curves and acts standardly on this union. This contradicts Lemma \ref{lem:curveaction}. As before, $\rho_0$ is the composition of $\rho:B_n\to \PMod(S_{g,p})$ with the forgetful map $\PMod(S_{g,p})\to \Mod(S_g)$.
\begin{proposition}
\label{prop:b}
Let $n \geq 26$ and $g \leq n-3$. Then, the labeled multicurve $\CRS(\rho_0(X_n))$ cannot contain curves of type I.
\end{proposition}

\begin{proof}

Let $X_n = \{\sigma_1,\dots,\sigma_{2k-1}\}$ and let $C_0$ be the set of components of $\CRS(\rho_0(X_n))$ with the trivial label $X_n$. This is exactly the set of components of type C.  For $1 \leq i \leq n$ odd, let $I_i$ be the set of components of $\CRS(\rho_0(X_n))$ with the label $X_n \setminus \sigma_i$.

We claim that, via $\rho_0$, every element of the group $B_n$ acts on the multicurve
\[
M = \left(\bigcup_{\sigma_i \in X_n} I_i\right) \cup C_0.
\]
The first step is to show that $\rho_0(\sigma_j)$ for each $\sigma_j \in X_n$ preserves the multicurve $I_i \cup C_0$ for each $i$.  For concreteness we consider the case $j=1$, but the other cases are proven in exactly the same way.  Since $\sigma_1$  commutes with each element of $X_n$, it follows that $\rho_0(\sigma_1)$ preserves $\CRS(\rho_0(\sigma_i))$ for each $\sigma_i \in X_n$.  Therefore, $\rho_0(\sigma_1)$ preserves the intersection of all $\CRS(\rho_0(\sigma_i))$, which is exactly $C_0$.  Similarly, $\rho_0(\sigma_1)$ preserves the intersection of all $\CRS(\rho_0(\sigma_j))$ for $j\neq i$.  These are exactly the sets $I_i \cup C_0$.  Since we already know that $\rho_0(\sigma_1)$ preserves $C_0$, it preserves each $I_i$.    This completes the first step.

The second step is to show that $\rho_0(\sigma_j)$ preserves $M$ for every even $j<n$, which implies that $\rho_0(B_n)$ preserves $M$.  We treat the case $j=2$, but the other cases are similar.  As above, since $\sigma_2$ commutes with $\sigma_5$ and $\sigma_7$, it follows that $\rho_0(\sigma_2)$ preserves the canonical reduction systems of both these elements, and thus preserves $\CRS(\rho_0(\sigma_5)) \cup \CRS(\rho_0(\sigma_7)) = M$, completing the proof of the claim.

By the claim, the group $\rho_0(B_n)$ acts on the set of components of $M$.  By the second statement of Lemma~\ref{lem:curveaction}, this action is cyclic. However, we claim that the action of $B_n$ on $M$ cannot be cyclic if $I_i$ is nonempty.  

We will prove this claim using the definition of totally symmetric sets. Since $X_n$ is totally symmetric, every permutation of $X_n$ can be achieved by conjugation by an element in $B_n$. Let $G$ be the stabilizer of $X_n$ inside $B_n$. The conjugation action gives a surjective homomorphism $G\to \Sigma_{|X_n|}$. As $G$ acts on $M$ by permuting labels, it follows that the resulting homomorphism from $G$ to the group of permutations of the set $\{I_1,I_3,\dots,I_{2k-1}\}$ is surjective. Thus, the action is not cyclic.
\end{proof}

%%%
%%%
%%%

\section{Case (5):  $\CRS(\rho_0(X_n))$ contains curves of type A but no curves of type $I$ 
}
\label{sec:pf}

In this section, we finish the proof of Theorem \ref{main1} by addressing the case when $\CRS(\rho_0(X_n))$ contains A-curves but no I-curves. We break the proof into three subsections. We first prove Theorem \ref{main1} for $\rho_0$, and then adapt our results to $\rho$ in Subsection 10.3.

Denote by $A_i$ the set of A-curves with label $\{\sigma_i\}$.  We will rule out all  but the following  two special cases in Subsection 10.1. The first possibility is that $A_i$ contains a separating curve. The other is that $A_i$ is a singleton consisting of a non-separating curve. These two possibilities will be discussed in Subsection 10.2 and 10.3 respectively. 

\subsection{\boldmath Topological types of A-curves in $\CRS(\rho_0(X_n))$}
In this subsection, we classify topological types of A-curves using totally symmetric sets. Let\[
Q= S_g-\bigcup_{i \text{ odd }<n}A_i.\] 
Firstly, we have the following property of components of $Q$.
\begin{lemma}\label{bigsmall}
Let $n>23$ and $g\le n-3$. Every component $R$ of $Q$ satisfies that either
\begin{itemize}
\item $\partial R \cap A_i\neq \emptyset$, for all odd $i<n$, and the multicurve $\partial R \cap A_i$ has the same topological type for each odd $i<n$; or
\item $\partial R\cap A_i \neq \emptyset$  for only a single $i$.
\end{itemize}
\end{lemma}
\begin{proof}
The proof is based on a computation of Euler characteristic. Let $R$ be a component of $Q$ with boundary not satisfying the statement. We label $R$ with the following subset of $X_n$:
\[
l(R)=\{\sigma_i\in X_n \ | \  \partial R\cap A_i\neq \emptyset.\}
\] Since $R$ does not satisfy the statement, either $|l(R)|\neq 1,k$ or $|l(R)|=k$ but  multicurves $\partial R \cap A_i$ do not have the same topological type. For $h\in \Mod(S_g)$ such that the action of $h$ by conjugation permutes $X_n$, we have that $h(R)$ is a component of $Q$ with boundary label $h(l(R))$.
\begin{itemize}
\item 
If $1< |l(R)|=l<k-1$, then we know that there are at least ${k \choose l}\ge {k \choose 2}$ different components of $Q$, as $X_n$ is totally symmetric. This contradicts the fact that the maximal number of components of  a partition of $S_g$ into subsurfaces is $2g-2$ and that $2g-2< {k \choose 2}$ in our range.
\item
If $|l(R)|=k-1$, we know that the Euler characteristic of $R$ satisfies
\[
\chi(R)\le 2-(k-1)=3-k,
\]
because $|\partial R|\ge k-1$. As $l(R)$ is a subset of a totally symmetric set, there are at least $k$ different components of $Q$ that are homeomorphic to $R$. Therefore, the Euler characteristic of $S_g$ satisfies
\[
2-2g= \chi(S_g)\le k(3-k).
\]
Again, in our range, this is not possible.
\item
Now we know that $|l(R)|=k$.  Let $T_i(R)$ be the topological type of $\partial R\cap A_i$. If $\{T_i(R)\}$ is not a singleton, then there exists a component $R'$ of $Q$ such that $\{T_i(R')\}$  is a permutation of $\{T_i(R)\}$. The number of such permutations is at least $k$. By a similar Euler characteristic computation as the previous two cases, we know that this is not possible, so all $A$ curves must have the same topological type.\qedhere
\end{itemize}
\end{proof}

Based on the above lemma, we call a component \emph{a small component} if its boundary only contains curves from a single $A_i$ and \emph{a big component} otherwise.
The following lemma describes the topological type of $A_i$.
\begin{lemma}\label{topologicaltype}
We have the following two possibilities:

(1) If one of the curves $a_i$ of $A_i$ is separating, then $a_i$ bounds a genus $1$ subsurface $H_i$. Furthermore, $A_i$ contains at most one separating curve.

(2) If $A_i$ contains no separating curves, then $A_i$ consists of a single non-separating curve.
\end{lemma}
\begin{proof} We will begin with the first statement.
\vskip 0.3cm
\noindent \emph{Statement (1).}\\
\indent We assume that $A_1$ contains separating curves. Let $a_1$ be a separating curve in $A_1$. Let $Y_1$ be the component of $S_g-a_1$ such that $g(Y_1)\le g(S_g-Y_1)$, where $g(-)$ denotes the genus. If $A_1$ contains more than one separating curve, we assume  in addition that $a_1\in A_1$ is the separating curve such that $g(Y_1)$ is maximal.

Let $R_1\subset Y_1$ be a component of $Q$ such that $a_1\in \partial R_1$. We claim that $R_1$ is a small component. Assume to the contrary that $R_1$ is a big component. By Lemma \ref{bigsmall}, the boundary of $R_1$ contains a separating curve $a_j$ of the same topological type in each $A_j$, for $1<j<n$ odd. We denote the smaller genus component of $S_g-a_j$ by $Y_j$.  However, since $a_j\subset Y_1$, we know that $Y_j$ is contained in $Y_1$, which is a contradiction because $a_j$ and $a_1$ have the same topological type, meaning that $Y_j$ and $Y_1$ have the same topological type. From the above argument, we also know that any component of $Q$ that is contained in $Y_1$ is a small component.

By total symmetry, we know that there exists a curve $a_i\in A_i$, for each $i < n$ odd, such that $a_i$ and $a_1$ have the same topological type. Let $Y_i$ be a component of $S_g-a_i$ with the same topological type as $Y_1$. Let $R$ be a big component of $Q$ whose boundary contains $a_1$. If $g(Y_1)>1$, then $\chi(Y_i)\le -3$. Since $R\cup \bigcup_{i \text{ odd}} Y_i\subset S_g$, the total Euler characteristic of $S_g$ satisfies 
\[
2-2g=\chi(S_g)\le \chi(R)+\sum_{i \text{ odd }<n}\chi(Y_i)\le 2-k+(-3k)=4-4k.
\]
However, by the assumption on $g$ and $n$, we also know that
\begin{equation}\label{eulernumber}
\chi(S_g)=2-2g\ge 2-2(n-3)\ge 2-2(2k+1-3)=6-4k.
\end{equation}
This is a contradiction. By the above discussion, if $A_i$ contains $2$ separating curves, then they both bound genus $1$ subsurfaces. Thus we know that $g\ge 2k$, which contradicts the fact $g\le n-3\le 2k+1-3=2k-2$.

\ \\ \emph{Statement (2).}\\
\indent If $A_i$ contains no separating curves but $A_i$ contains more than one curve, then we break the discussion into cases depending on the existence of small components of $Q$.

\begin{itemize}

\item \emph{Case 1: $Q$ does not contain a small component.}\\ 
In this case, every component is a big component with Euler characteristic at most $2-k$, since it has at least $k$ boundary components. Let $R_1,...,R_r$ be connected components of $Q$. Let $e_i$ be the number of boundary components of $R_i$. Thus the maximal number of big components $r$ is $4$ by Inequality \eqref{eulernumber}. Assume that $A_i$ contains $l$ curves. Then we have that
\[
e_1+...+e_r = 2lk,
\]
thus
\[
\chi(S_g)\le 2-e_1+...+2-e_r\le 2r-(e_1+...+e_r)\le 2r-2lk.
\]
By Inequality \eqref{eulernumber}, we know that $l\le 2$. Since $l\neq 1$, we know that $3 \le r\le 4$. 
\begin{itemize}
\medskip \item \textsc{Case A: $r=4$. }
\\ In this case, $|\partial R_i| = k$ for every $1\le i\le r$ because otherwise $|\partial R_i| \ge 2k$, which contradicts Inequality \eqref{eulernumber}. We claim that either $\partial R_2\cap \partial R_1=\emptyset$ or $\partial R_1=\partial R_2$.

If $\partial R_1$ and $\partial R_2$ have nontrivial intersection, by the action on the labels, we know there are at least $k$ components of $Q$, which is a contradiction. Without loss of generality, we assume $\partial R_1 = \partial R_2$. This means that $R_1\cup R_2$ is a closed surface. This is a contradiction.

\medskip \item \textsc{Case B: $r=3$.}
\\ In this case, we know that one of $e_i$ is $2k$. Assume that $e_1=2k$ and $e_2=e_3=k$. Let $A_i=\{a_i,a_i'\}$ for $i<n$ odd and the union of $A_i$ bound a subsurface $R_1$. The Euler characteristic is
\[
\chi(S_g)=2-2k+2-k+2-k=6-4k.
\]
By Inequality \eqref{eulernumber}, this only happens when $n=2k+1$ and $g=n-3$. In this case, we know that 
\[
X= \{\sigma_1,\sigma_3,...,\sigma_{2k-3},\sigma_{2k}\}\]
is another totally symmetric set of $B_n$. By the same argument, we know that there exists $A_{2k} = \{a_{2k}, a_{2k}'\}$ such that $\{A_1,A_3,..., A_{2k-3},A_{2k}\}$ satisfies the same property as $\{A_1,A_3,...,A_{2k-1}\}$. Let $[a_i], [a_i'] \in H_1(S_g;\mathbb{Z}/2)$ be the homology classes of curves $a_i,a_i'$. We use $\mathbb{Z}/2$-coefficients so that we do not need to consider orientation of curves. Since the boundary of $R_1$ is exactly the set of all the A-curves, we have 
\[
[a_1]+[a_1']+[a_3]+[a_3']+...+[a_{2k-1}]+[a_{2k-1}']=0
\]
and similarly, we have
\[
[a_1]+[a_1']+[a_3]+[a_3']+...+[a_{2k-3}]+[a_{2k-3}']+[a_{2k}]+[a_{2k}']=0.
\]
The subtraction of the above two equations gives
\[
[a_{2k-1}]+[a_{2k-1}']=[a_{2k}]+[a_{2k}'].
\]
By symmetry, we know that 
\[
[a_{1}]+[a_{1}']=[a_{2}]+[a_{2}']=[a_{3}]+[a_{3}']
\]
This is not possible since otherwise $\{a_1,a_1',a_3,a_3'\}$ separates the surface $S_g$, which contradicts Lemma \ref{bigsmall}.\qedhere

\end{itemize}

\medskip \item \emph{Case 2: $Q$ contains small components.}\\
We assume $Q_i$ is the union of small components of $Q$ with label $\{\sigma_i\}$. Then $|\partial Q_i| \ge 2$ because $A_i$ does not contain separating curves. If $|\partial Q_i|= 2$, then $Q_i$ is connected and the genus $g(Q_i)\ge 1$. Let $R_1,...,R_r$ be all the big components of $Q$. Each of $R_i$ has Euler characteristic at most $2-k$ since it has at least $k$ boundary components. Then the Euler characteristic of $S_g$ is at most $r(2-k)$. The Euler characteristic of $S_g$ also satisfies Inequality \eqref{eulernumber}, which implies that $r\le 4$. Let $e_i$ be the number of boundary components of $R_i$. 
\begin{itemize}
\medskip \item \textsc{Case A: $|\partial Q_i|=m > 2$.}
\\ In this case, we know that $e_1+...+e_r\ge mk$. Then the Euler characteristic satisfies
\[
\chi(S_g)\le \sum_{1\le j\le r} \chi(R_j)+\sum_{i<n\text{ odd}} \chi(Q_i)\le 2m-\sum_{1\le j\le r} e_j+k(2-m)\le 2m+(2-2m)k.
\]
We then know that $m=3$ and $r=3$ from Inequality \eqref{eulernumber}, so Inequality \eqref{eulernumber} becomes equality. Then $n = 2k+1$ and $g = n-3$. We apply the same proof as Case 1 Case B to conclude that this is not possible.

\medskip \item \textsc{Case B: $|\partial Q_i| = 2$ and $g(Q_i)\ge 1$.}
\\
In this case, we know that $e_1+...+e_r\ge 2k$. If $r\le 2$, then the Euler characteristic satisfies
\[
\chi(S_g)\le \sum_{1\le j\le r} \chi(R_j)+\sum_{i<n\text{ odd}} \chi(Q_i)\le 2r-\sum_{1\le j\le r} e_j-2k\le 4-2k-2k=4-4k,
\]
which contradicts  Inequality \eqref{eulernumber}. If $r>2$, then
\[
\chi(S_g)\le \sum_{1\le j\le r} \chi(R_j)+\sum_{i<n\text{ odd}} \chi(Q_i)\le r(2-k)-2k\le 4-2k-2k\le 3(2-k)-2k=6-5k,
\]
which also contradicts  Inequality \eqref{eulernumber}.
\end{itemize}

\end{itemize}
\end{proof}

\subsection{\boldmath Ruling out separating curves in $A_i$} In this subsection, we will rule out the existence of separating curves in $A_i$ by a homology argument. Notice that when $g=n-2$, we have McMullen's example, which satisfies that the A-curve is a single separating curve which bounds a genus $1$ subsurface \cite{McMullen}. Therefore, the bound on the genus is sharp in the following proof.

By Lemma \ref{topologicaltype}, a separating curve $a_i \in A_i$ divides the surface $S_g$ into a genus $1$ subsurface $Y_i$ and a genus $g-1$ subsurface. The proof has three steps. We first show that there is a component $S'$ of $S_g-M$ that contains $a_i$ for all $i$. Then we show that there is a finite order element $\tau$ that is equal to $\rho_0(\sigma_i)$ on the subsurface $S'$, for every $i$. Lastly, by transvecting $\rho_0$ with $\tau$ on $S'$, we use a homology argument to show that $Y_2,...,Y_{n-2}$ are independent subspaces, which contradicts the bound on $g$ and $n$.

We first discuss the relationship between C-curves and A-curves. Let $M$ be the set of $C$ curves and let $a_i$ be the unique separating curve in $A_i$.
\begin{lemma}\label{haha}
There is a component $S'$ of $S_g-M$ containing every $a_i$ for $1\le i<n$. Furthermore, $S'$ is $\rho_0(B_n)$-invariant.
\end{lemma}
\begin{proof}
Let $S(i)$ be the component of $S_g-M$ such that $a_i\in S(i)$. We claim that $S(1)=S(3)$, which by symmetry implies that $S':=S(1)=S(i)$ for all $1\le i<n$. If $S(1)\neq S(3)$, then by symmetry, we know that $S(1),S(3),...,S({2k-1})$ are pairwise different. Furthermore, we know $S(1),S(2),...,S({n-1})$ are pairwise different by symmetry. Then either $\chi(S(1))=-1$ or $\chi(S(1))\le -2$. If $\chi(S(i))\le -2$, we know that 
\[
\chi(S_g)\le \sum_{1\le i<n} \chi(S_i)\le -2(n-1)\le 2-2(2k)=2-4k,
\]
which contradicts Inequality \eqref{eulernumber}. Then we know that $\chi(S(1))=-1$, which implies that $S(1)$ is either a once-punctured torus or a pair of pants. 
If $S(1)$ is a once-punctured torus, then the genus $g\ge n-1$ because there are $n-1$ disjoint once-punctured subtori $S(1),S(2),...,S({n-1})$  in $S_g$. If $S(1)$ is a pair of pants, then $S(1)$ does not contain a nonperipheral curve $a_1$.

\end{proof}
Let $S'$ be the surface as in Lemma \ref{haha} and let \[
S=S'\cup \bigcup_{i <n}Y_{i}.
\]
 Let 
\[R= S-\bigcup_{i\text{ odd }<n}Y_{i},\] 
where
 \[
g(R)\le n-3-k\le 2k+1-3-k=k-2.\] Let $Y_i^c = S- Y_i$. We first show the following.
\begin{claim}
$A_i$ contains no curves inside $Y_i^c$.
\end{claim}
\begin{proof}
Firstly by Lemma \ref{bigsmall}, we know that it is not possible for $A_1$ to contain any curve in $Y_j$ for $1<j<n$ odd. If $A_1$ contains a nontrivial subset $A_1'$ consisting of non-peripheral curves in $R$, then $A_i$ contains a nontrivial subset $A_i'$ consisting of non-peripheral curves in $R$ for $i<n$ odd by symmetry. Then 
\[R-\bigcup_{i\text{ odd }<n}A_i'\]
is a union of big components $R_1,...,R_r$. Without loss of generality, assume that $\partial R_1$ contains $a_i$ for every $i<n$ odd, where the existence follows from symmetry. Then the boundary of $R_1$ has at least $2k$ curves. The Euler characteristic of $R$ satisfies
\[
\chi(R)=\chi(R_1)+...+\chi(R_r)\le 2-2k
\]
However, we know that $\chi(R)=2-2g(R)\ge 2-2(k-2)=6-2k$, which is a contradiction.
\end{proof}

We now prove the main lemma about the behavior of $\rho_0(\sigma_i)$ on $Y_i^c$. Recall that $Y_i^c$ is the complement of $Y_i$ in $S$.
\begin{lemma}\label{transvection}
There exists a finite order element $\tau\in \Mod(S)$ such that
\begin{itemize}
\item $\tau$ commutes with $\rho_0(\sigma_i)|_S$ for every $i$, and
\item $\tau^{-1}\rho_0(\sigma_i)|_{S}$ is the identity on $Y_i^c$  for $0<i<n$
\end{itemize}
\end{lemma}
\begin{proof}
We break the proof into the following three steps: 

\p{\boldmath Step 1: We claim that  $\rho_0(\sigma_1)$ and $\rho_0(\sigma_3)$ are equal on $S-Y_1-Y_3$}\vskip 0.3cm

Since $\rho_0(\sigma_1)$ has empty canonical reduction system on $Y_1^c$, we know that $\rho_0(\sigma_1)$ is either periodic or pseudo-Anosov on $Y_1^c$. Since $\rho_0(\sigma_1)$ commutes with $\rho_0(\sigma_3)$, implying that $\rho_0(\sigma_1)$ preserves $a_3$, we know that $\rho_0(\sigma_1)$ on $Y_1^c$ cannot be pseudo-Anosov. Therefore the restrictions $\{\rho_0(\sigma_i)|_R\}$ are a totally symmetric set of torsion elements in $\Mod(R)$. This implies that  \[ \rho_0(\sigma_1)= \rho_0(\sigma_3)\text{ on $R$}\]
by Proposition \ref{prop:mcgtorsion}.

Then on $S-Y_1-Y_3$, the element $\phi = \rho_0(\sigma_1)\rho_0(\sigma_3)^{-1}$ is a finite order element which is the identity on the subsurface $R$. The following claim concludes that $\phi$ is the identity on $S-Y_1-Y_3$.

\begin{claim}\label{finiteorder}
A nontrivial finite order mapping class cannot be the identity on a subsurface.
\end{claim} 
\begin{proof}
Any nonidentity finite order mapping class $f$ of the surface $F$ has a geometric realization by isometry by Nielsen and Fenchel (see, e.g., \cite[Chapter 7]{FM}) and the realization is unique up to conjugacy (Ahlfors' trick). On any subsurface, the realization is not the identity because a nonidentity isometry cannot be identity on an open subset. The uniqueness of geometric realizations implies that $f$ cannot be the identity on any subsurface. 
\end{proof}

\p{\boldmath Step 2: There exists a unique finite order element $\tau\in \Mod(S)$ such that 
\[
\tau|_{Y_1^c}=\rho_0(\sigma_1)|_{Y_1^c}\text{ and }\tau|_{Y_3^c}=\rho_0(\sigma_3)|_{Y_3^c}.\]}\indent
We construct $\tau$ by geometric realizations of $\rho_0(\sigma_1)$ and $\rho_0(\sigma_3)$ that are equal on $S-Y_1-Y_3$. We construct the element $\tau$ from the following claim.
\begin{claim}
Let  $f\in\Mod(F)$ be a mapping class and $F'\subset F$ be a subsurface such that $f(F')=F'$. Let $\phi'$ be a homeomorphism of $F'$ preserving the boundary components of $F'$ as a set, so that $\phi'$ is homotopic to $f$ over $F'$. Then, there exists a homeomorphism $\phi$ of $F$ such that $\phi$ is homotopic to $f$ and $\phi|_{F'}=\phi'$.
\end{claim}
\begin{proof}
Firstly, we extend $\phi'$ to a map $\phi''$ defined on $F$ such that $\phi''$ is equal to $\phi'$ on $F'$ and $\phi''$ is homotopic to $f$ on $F-F'$. We now see that as a mapping class, $\phi''f^{-1}$ is identity on $F'$ and $F-F'$. Then as a mapping class, $\phi''f^{-1}$ is a multitwist about curves in $\partial F'$. Then by multiplying the corresponding multitwist supported on subannuli to $\phi''$ outside of $F'$, we obtain a homeomorphism $\phi$ that is homotopic to $f$ and equal to $\phi'$ on $F'$.
\end{proof}
The uniqueness follows from the next claim.
\begin{claim}
If $h\in \Mod(S)$ satisfies that $h|_{Y_1^c}=id$ and $h|_{Y_3^c}=id$, then $h=id$.
\end{claim}
\begin{proof}
We know that $\CRS(h)$ is empty from the assumption. As $h$ cannot be pseudo-Anosov, $h$ is finite order. By Claim \ref{finiteorder}, we know that $h$ is the identity element.
\end{proof}
We claim that $\tau$ is a finite order element. If not, we consider the canonical reduction system of $\tau$. By the definition of $\tau$ and the fact that $\rho_0(\sigma_1)|_{Y_1^c}$ and $\rho_0(\sigma_3)|_{Y_3^c}$ are finite order elements, we know that $\CRS(\tau)$ cannot contain any curve in both $Y_3^c$ and $Y_1^c$. We know that $\CRS(\tau)$ is empty but $\tau$ cannot be pseudo-Anosov. This implies that $\tau$ is a finite order element.

\p{\boldmath Step 3: $\tau$ also satisfies that 
\begin{equation}\label{=}
\tau|_{Y_i^c}=\rho_0(\sigma_i)|_{Y_i^c} \text{ for $i<n$}
\end{equation}}
\indent
Firstly, we prove that Equation \eqref{=} is true for $i$ odd. For $\sigma_3$ and $\sigma_5$, we define a similar element called $\tau_{35}$ as in Step 2. To prove the claim, we only need to show that $\tau=\tau_{35}$. 

By construction, we know that $\tau$ commutes with $\tau_{35}$. We also know that $\tau|_R=\tau_{13}|_R$. Then as a finite order element $\tau_{35}\tau^{-1}$ is the identity on $R$. This implies that $\tau=\tau_{35}$ by Claim \ref{finiteorder}. 

Similarly for any pair of numbers $2\le  i<j<n$ such that $j-i\ge 2$, we construct an element $\tau_{ij}$ that satisfies the condition in Step 2 for $\sigma_i,\sigma_j$. By the above argument, we have $\tau= \tau_{13}=\tau_{1(j+2)}=\tau_{j(j+2)}=\tau_{ij}$. This implies that $\tau$ satisfies Equation \eqref{=}.

\end{proof}

We now finish the proof of Theorem \ref{main1} in the case when $A_i$ contains separating curves.
Now, a transvection of the original $\rho_0|_S$ by $\tau$ gives  us a new homomorphism. We call the new homomorphism $\rho_0'$, with image in $\Mod(S)$. By Lemma \ref{transvection}, $\rho_0'$ satisfies $\rho_0'(\sigma_i)=id$ on $Y_i^c$. Let $\overline{Y_i}\subset H_1(S_g;\mathbb{Q})$ be the subspace spanned by curves in $Y_i$. Now the following claim implies that the dimension of $H_1(S_g;\mathbb{Q})$ is $2g$ but is also at least $2(n-2)$. We conclude that $g\ge n-2$, which is a contradiction.
\begin{claim}
$\overline{Y_2},...,\overline{Y_{n-1}}$ are independent subspaces.
\end{claim}
\begin{proof}
Assume that there is a relation between them:
\begin{equation}\label{E1}
x_k+\cdots+x_{n-1}=0
\end{equation}
such that $x_i\in \overline{Y_i}$, $k>1$ and $x_k\neq 0$. We apply $\rho_0'(\sigma_{k-1})$ on the above relation, we obtain a new relation
\begin{equation}\label{E2}
\rho_0'(\sigma_{k-1})(x_{k})+...+x_{n-1}=0
\end{equation}
By subtracting the two equations \eqref{E1}, \eqref{E2}, we obtain
\[
x_k-\rho_0'(\sigma_{k-1})(x_k)=0
\]
Since $\sigma_{k-1}\sigma_k\sigma_{k-1}^{-1}$ is conjugate to $\sigma_k$, we know that $\rho_0(\sigma_{k-1})(Y_k)$ is the genus $1$ subsurface corresponding to the element $\sigma_{k-1}\sigma_k\sigma_{k-1}^{-1}$.
If 
\begin{equation}\label{int}
\rho_0'(\sigma_{k-1})(\overline{Y_{k}})\cap \overline{Y_k}\neq \{0\}, \end{equation}
by the braid relation
\[
\sigma_k (\sigma_{k-1} \sigma_k \sigma_{k-1}^{-1}) \sigma_k^{-1} = \sigma_{k-1} \text{ and } \sigma_k \sigma_k \sigma_k^{-1}= \sigma_k,
\]
we know that 
\[
\rho_0'(\sigma_k)\rho_0'(\sigma_{k-1})(\overline{Y_{k}}) = \overline{Y_{k-1}}\text{ and } \rho_0'(\sigma_k)(\overline{Y_k}) = \overline{Y_k}
\]
By applying $\rho_0'(\sigma_k)$ to Equation \eqref{int}, we obtain
\[
\overline{Y_{k-1}} \cap\overline{Y_k}\neq \{0\} \]
By symmetry, we know that 
\[
\overline{Y_1} \cap\overline{Y_2}\neq \{0\} \text{ and } \overline{Y_{2}} \cap\overline{Y_3}\neq \{0\}
\]
However, nontrivial elements in $\overline{Y_1}$ and $\overline{Y_3}$ are independent and $\overline{Y_2}$ is 2-dimensional, we know that \[\overline{Y_2}\subset \overline{Y_1}+ \overline{Y_3}\]
This is a contradiction because two elements in $\overline{Y_1},\overline{Y_3}$ have zero intersection number but $\overline{Y_2}$ has nontrivial intersection pairing.
\end{proof}

\subsection{$A_i$ consists of a single non-separating curve}
In this case, we will prove that $\rho$ is a transvection of the standard homomorphism. The proof has three steps. We first show that the canonical reduction system of $\rho(X_n)$ only has A-curves and C-curves, and the A-curve is a single non-separating curve. Secondly, we show that there is an element $\tau$ in $\PMod(S_{g,p})$ such that $\tau^{-1}\rho(\sigma_i)=T_{a_i}^k$ for some $k$ and every $i$, where $T_{a_i}$ is the Dehn twist about the curve $a_i$. Thirdly, we conclude that $k$ is $\pm 1$ by the braid relation.

To study $\rho$, we need to analyze $\CRS(\rho(X_n))$. We classify curves of $\CRS(\rho(\sigma_i))$ by looking at their representative in $\CRS(\rho_0(\sigma_i))$. The canonical reduction system $\CRS(\rho_0(\sigma_i))$ can be obtained from $\CRS(\rho(\sigma_i))$ by forgetting punctures on the underlying surface, identifying isotopic curves, and deleting peripheral curves. We call a curve $c_0$ in $\CRS(\rho_0(\sigma_i))$ the \emph{forgetful curve} of a curve $c \in \CRS(\rho(\sigma_i))$ if $c_0$ is the corresponding curve on $S_g$ obtained by forgetting punctures.

The following lemma describes the relationship between curves in $\CRS(\rho_0(\sigma_i))$ and $\CRS(\rho(\sigma_i))$. Recall from Propositions \ref{prop:abc} and \ref{prop:b} that a curve in $\CRS(\rho_0(X_n))$ must be either of type $A$ or type $C$.
\begin{lemma}
Let $c_0 \in \CRS(\rho_0(\sigma_i))$ be the forgetful curve of  $c \in \CRS(\rho(\sigma_i))$.
\begin{enumerate}
\item If $c_0$ is a type A curve, $c$ is also a type A curve in $\CRS(\rho_0(X_n))$. 
\item If $c_0$ is peripheral, then $c$ is a type C curve in $\CRS(\rho_0(X_n))$.
\item If $c_0$ is a type C curve, then $c$ is a type C curve in $\CRS(\rho_0(X_n))$.
\end{enumerate}
\end{lemma}
\begin{proof}\ 
\begin{enumerate}
\item If $c_0$ is not a type $A$ curve, then $c$ would be in $\CRS(\rho(\sigma_i))$ for more than one $i$, contradicting the fact that $c$ is a type $A$ curve.

\item If $c$ is not a type $C$ curve, then $|l(c)|\neq k$. For an element $h\in B_n$ such that $h(B_n)h^{-1}=B_n$, we know that  $\rho_0(h)(c)$ has label $h(l(c))$. If $h(l(c))\neq  l(c)$, then $\rho_0(h)(c)$ and $c$ are disjoint. This is not possible because elements in $\PMod(S_{g,p})$ do not permute punctures.

\item If $c$ is not a type C curve, then $|l(c)|\neq k$. Since $c_0$ is a type C curve, there is a curve $c'\neq c_0\in \CRS(\rho_0(X_n))$ such that $c_0$ and $c'$ are isotopic. Since $c,c'$ are isotopic after forgetting punctures, we know that $S_{g,p}-c-c'$ has a component that is a punctured annulus $\mathcal{A}(c,c')$. Since $c$ is not a type C curve, there exist at least two elements $h_1,h_2\in B_n$ such that $h_i(X_n)=X_n$ and $l(c), h_1(l(c)),h_2(l(c))$ are pairwise different. 

\noindent Since $c, \rho(h_1)(c)$, and $\rho(h_2)(c)$ have different labels, we know that they are pairwise different. We know then either that $\rho(h)(\mathcal{A}({c,c'}))$ is disjoint from $\mathcal{A}({c,c'})$ or that $\rho(h_i)(\mathcal{A}({c,c'}))=\mathcal{A}({c,c'})$. However, elements in $\PMod(S_{g,p})$ do not permute punctures. Thus $\rho(h_i)(\mathcal{A}({c,c'}))=\mathcal{A}({c,c'})$. That means $\rho(h_1)(c)=\rho(h_2)(c)=c'$, which is a contradiction.\qedhere
\end{enumerate}
\end{proof}

The above lemma implies that curves in $\CRS(\rho(X_n))$ are either type A or type C. We now prove the following claim about the behavior of $\rho(\sigma_i)$. Let $A_{i}$ be the set of $A$ curves with label $\sigma_i$ in $\CRS(\rho(X_n))$. 
\begin{lemma}
The set $A_{i}$ consists of a single non-separating curve.
\end{lemma}
\begin{proof}
The forgetful curve of $A_{i}$ must be a single non-separating curve. If $A_{i}$ contains more than one curve, then elements in $A_{i}$ are all isotopic to each other after forgetting punctures. Let $c_1,c_1'$ be two different curves in $A_{1}$. Since $c_1,c_1'$ are isotopic after forgetting punctures, there exists a punctured annulus $A(c_1,c_1')$ with boundary $c_1,c_1'$. By symmetry, there is $h\in \rho(B_n)$ such that $h(A_{1})=A_{3}$. Then, we know that $h(A(c_1,c_1'))$ is disjoint from $A(c_1,c_1')$. This is not possible because elements in $\PMod(S_{g,p})$ fix punctures pointwise.
\end{proof}

We now discuss the relationship between type C curves and type A curves of $\CRS(\rho(X_n))$. Let $M$ be the set of $C$ curves of $\CRS(\rho(X_n))$ and let $a_i$ be the single non-separating curve in $A_{i}$.
\begin{lemma}\label{haha2}
There is a component $S$ of $S_{g,p}-M$ that contains every $a_i$ for $1\le i<n$. Furthermore, $S$ is $\rho(B_n)$-invariant.
\end{lemma}
\begin{proof}
Let $S(i)$ be the component of $S_{g,p}-M$ such that $a_i\in S(i)$. If $S(1)\neq S(3)$, then there is an element $h\in \rho(B_n)$ such that $h(S(1))=S(3)$. Thus $S(1)$ does not contain any punctures because elements in $\PMod(S_{g,p})$ fixes punctures pointwise. The remainder of the proof follows Lemma \ref{haha} exactly. \qedhere

\end{proof}
Let $S$ be the subsurface of $S_{g,p}$ that is obtained from Lemma \ref{haha2}. We now prove the main lemma.
\begin{lemma}\label{transvection2}
There exists an element $\tau\in \PMod(S_{g,p})$ such that
\begin{itemize}
\item $\tau$ commutes with $\rho(\sigma_i)$ for every $i$, and
\item there exists an integer $k$ such that $\tau^{-1}\rho(\sigma_i)=T_{a_i}^k$, for $0<i<n$
\end{itemize}
\end{lemma}
\begin{proof}
The proof is very similar to the proof of Lemma \ref{transvection}. We break the proof into several steps.
\begin{itemize}
\item
\medskip \textsc{Step 1:} \emph{We claim that $\rho(\sigma_1)$ and $\rho(\sigma_3)$ are equal on $S_{g,p}-a_1-a_3$. } \\
By Lemma \ref{haha}, we know that $S$ contains all type $A$ curves. Therefore $S$ is invariant under $\rho(B_n)$. We obtain new representations $\rho_1,\rho_2$ from $\rho$ by restricting to $S$ and $S_{g,p}-S$ respectively.
\[
(\rho_1,\rho_2): B_n\to \Mod(S)\times \Mod(S_{g,p}-S)
\]
We break the proof into three steps.
\begin{itemize}
\item
\medskip \textsc{Substep 1:} \emph{$\rho(\sigma_1\sigma_3^{-1})=id$ on $S_{g,p}-S$.} \\
This follows from Proposition \ref{prop:c} because on $S_{g,p}-S$, all elements in $\CRS(\rho_2(X_n))$ are C-curves.

\item
\medskip \textsc{Substep 2:} \emph{$\rho(\sigma_1\sigma_3^{-1})=id$ on $S-a_1-a_3$.}
\\ The proof is the same as Step 1 of Lemma \ref{transvection}.
\item 
\medskip \textsc{Substep 3:} \emph{$\rho(\sigma_1\sigma_3^{-1})=id$ on $S_{g,p}-a_1-a_3$.}
\\ From the first two substeps,  we know that a power of $\rho(\sigma_1\sigma_3^{-1})$ is a multitwist about C-curves on $S_g-a_1-a_3$. Similarly,  a power of $\rho(\sigma_1\sigma_5^{-1})$ is a multitwist about C-curves on $S_g-a_1-a_5$. We know $\sigma_1\sigma_3^{-1}$ and $\sigma_1\sigma_5^{-1}$ are conjugate to each other by an element $h\in B_n'$ and the action of $B_n$ on the set of C-curves is cyclic by Substep 1. This implies that on $S_g-a_1-a_3-a_5$, up to a power, we have $\rho(\sigma_1\sigma_3^{-1})=\rho(\sigma_1\sigma_5^{-1})$. This shows $\rho(\sigma_3\sigma_5^{-1})$ is a periodic element on $S_g-a_1-a_3-a_5$. However by Substep 2, $\rho(\sigma_3\sigma_5^{-1})$  is identity on $S-a_3-a_5$, which implies that  $\rho(\sigma_3\sigma_5^{-1})$ is identity on $S_g-a_1-a_3-a_5$ by Claim \ref{finiteorder}. This further shows that $\rho(\sigma_3)=\rho(\sigma_5)$ on $S_g-a_3-a_5$.  By symmetry, we obtain the claim of Step 1.
\end{itemize}

\item
\medskip \textsc{Step 2:} \emph{There exists $\tau\in \PMod(S_{g,p})$ such that
\[
\tau|_{S_{g,p}-a_1}=\rho(\sigma_1)|_{S_{g,p}-a_1}\text{ and }\tau|_{S_{g,p}-a_3}=\rho(\sigma_3)|_{S_{g,p}-a_3}.\]
}
\item
\medskip \textsc{Step 3:} \emph{$\tau$ also satisfies that 
$$ \tau|_{S_{g,p}-a_i}=\rho(\sigma_i)|_{S_{g,p}-a_i} \text{ for $i<n$} $$}
\end{itemize}
The proofs of Steps 2 and 3 are the same as the proofs of Steps 2 and 3 of Lemma \ref{transvection} and are omitted. Since $\tau^{-1}\rho(\sigma_i)$ is identity on $S_{g,p}-a_i$, we know that $\tau^{-1}\rho(\sigma_i)$ is a Dehn twist along $a_i$. Thus the lemma follows.
\end{proof}
We now finish the proof of Theorem \ref{main1} when the type A curve of $\CRS(\rho(X_n))$ is a single non-separating curve. By Claim \ref{transvection2}, a transvection of the original $\rho$ by $\tau$ gives us a new representation $\rho'$ which satisfies that $\rho'(\sigma_i)=T_{a_i}^k$. Since $\sigma_i$ and $\sigma_{i+1}$ satisfy the braid relation, we know that $T_{a_i}^k$ and $T_{a_{i+1}}^k$ satisfy the braid relation. McCarthy \cite{braidrelation} finds that if $T_a^{k}$ and $T_b^{k}$ satisfy the braid relation, then $k=\pm 1$ and $i(a,b)=1$, which finishes the proof. We then know that up to a transvection, the original representation is either $\rho_s$ or $\rho_{-s}$.

\section{How Theorem \ref{main1} implies Corollary \ref{maincor}}

Let $\phi: \PMod(S_{g,p})\to \PMod(S_{h,q})$ be a nontrivial homomorphism for $h\le 2g-1$. We only sketch the proof of \cite[Proposition 1.6, Chapter 1-8]{AS}, which states that $\phi$ maps a (right) Dehn twist along a non-separating curve to a (possibly left) 
Dehn twist along a non-separating curve. For how to use   \cite[Proposition 1.6]{AS} to prove Corollary \ref{maincor}, we refer the reader to \cite[Chapter 9-10]{AS}.

\begin{proposition}
Let $g>23$ and $h\le 2g-1$. Let $\phi: \PMod(S_{g,p})\to \PMod(S_{h,q})$ be a nontrivial homomorphism. Then $\phi$ maps a (right) Dehn twist along a non-separating curve to a (possibly left)  Dehn twist along a non-separating curve. 
\end{proposition}
\begin{proof}
Let $i: B_{2g+2}\to \PMod(S_{g,p})$ be the standard embedding which maps generators $\sigma_i \in B_{2g+2}$ to the Dehn twists $T_{c_i} \in \PMod(S_{g,p})$, where the curves $c_i$ form a chain, i.e., $i(c_i,c_{i+1}) = 1$ and $i(c_i,c_j) = 0$ for $|i-j| > 1$. Let
\[
\rho:=\phi\circ i: B_{2g+2}\to \PMod(S_{g,p})\to \PMod(S_{h,q})
\]
be the composition. By Theorem \ref{main1}, the composition $\rho$ is either standard or has cyclic image up to transvection. We break the following discussion into cases depending on $\rho$.

\p{\boldmath Case 1: $\rho$ has cyclic image.}\\
\indent
If $\rho$ is cyclic, then $\rho(B_{2g+2}')$ is trivial. We aim to show that $\phi$ must be trivial. As the Dehn twist $T_c$ for $c$ a non-separating curve normally generates $\Mod(S_g)$, it suffices to show $\phi(T_c)=1$ for some non-separating curve $c$. Since $\rho(\sigma_1\sigma_3^{-1})=1$, we know that $\phi(T_{c_1}T_{c_3}^{-1})=1$. Since $c_1,c_3$ are both non-separating and their union is not the boundary of any subsurface, we know by conjugation that $\phi(T_cT_d^{-1})=1$ for any pair of curves $c,d$ satisfying that both are non-separating and do not bound a subsurface.

We will make use of the lantern relation (see e.g., \cite[Chapter 5]{FM}):
\begin{equation}\label{lantern}
T_c=T_{a_1}T_{b_1}^{-1}T_{a_2}T_{b_2}^{-1}T_{a_3}T_{b_3}^{-1}.
\end{equation}
where $c,b_1,b_2,b_3$ bound a sphere with four boundary components $F=S_{0,4}$ and $a_1,a_2,a_3$ are three curves on $F$. We can position all seven curves as non-separating curves such that no pair of them bounds a subsurface. We apply $\phi$ to this lantern relation to obtain $\phi(T_c)=1$. This concludes the proof for Case 1.

\p{\boldmath Case 2: $\rho$ is standard up to transvection.}
\\ \indent Now \[
\rho(\sigma_i) = \phi(T_{c_i}) = T_{\overline{c_i}} \cdot \tau\] where $\overline{c_i}$ is non-separating, $\{\overline{c_i}\}$ is a chain of curves in $S_{h,q}$, and $\tau\in\PMod(S_{h,q})$ commutes with $\rho(B_{2g+2})$. We aim to show that that $\tau$ is the identity. We claim that $\tau$ commutes with every element in $\phi(\PMod(S_{g,p}))$. To prove this, we make use of the Humphries generating set \cite{Humphries}. 
\begin{figure}[H]
\includegraphics[scale=0.8]{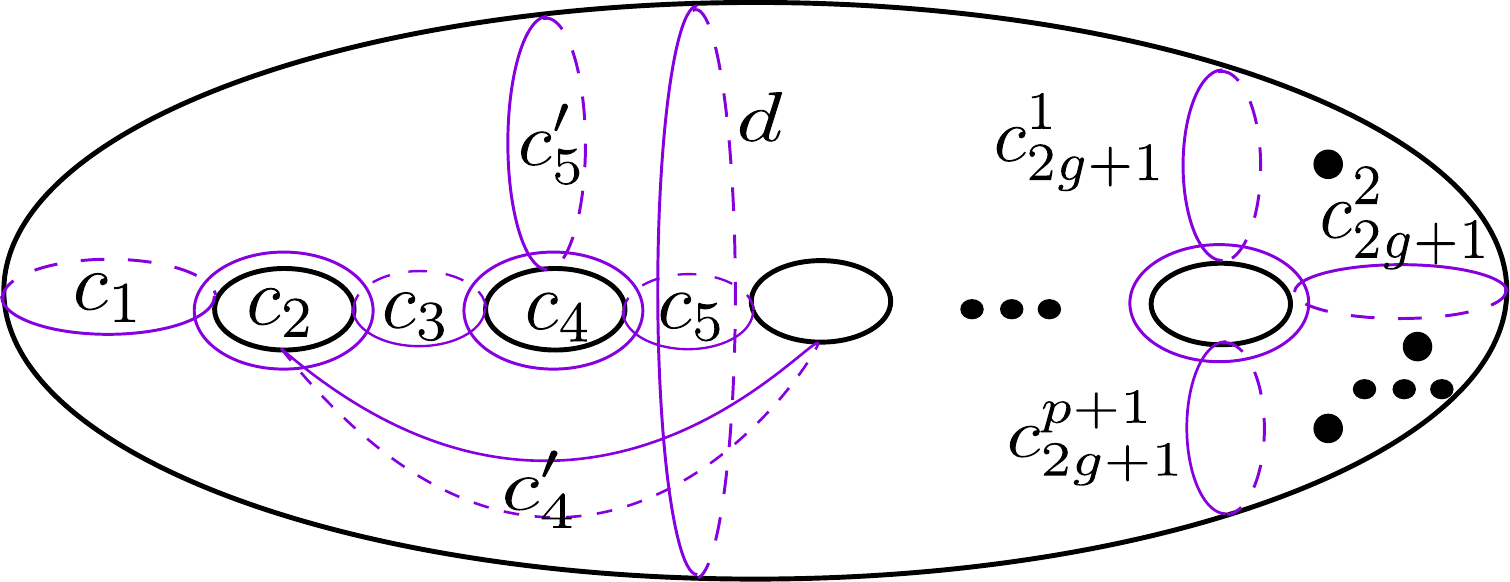}
\caption{Humphries generating curves}
\label{generatingset}
\end{figure}
We know that \[\{
c_1,...,c_{2g},c_{2g+1}^m\}\]is a chain of $2g+1$ curves for any $1\le m\le p+1$. Therefore, we know that $\tau$ commutes with $\phi(T_{c_{2g+1}^m})$ for any $1\le m\le p+1$. Since 
\[\{c_5',c_4, c_3, c_4', c_6,c_7,...,c_{2g},c_{2g+1}^1\}\]
is also a chain of curves of length $2g+1$, we know in particular  that $\tau$ commutes with $\phi(T_{c_5'})$. Thus, we know that $\tau$ commutes with $\phi(\PMod(S_{g,p}))$ because Dehn twists along 
\[\{c_1,...,c_{2g},c_{2g+1}^1,...,c_{2g+1}^{p+1},c_5'\}\] generate the group $\PMod(S_{g,p})$ by \cite{Humphries}. 

We know that $\tau$ commutes with $\phi(\PMod(S_{g,p}))$. By the chain relation (see, e.g., \cite[Chapter 9]{FM}), we know the curves in Figure \ref{generatingset} satisfy
\[
T_d=(T_{c_1}T_{c_2}T_{c_3}T_{c_4})^5.
\]
By applying $\phi$ to both sides, we obtain
\[
\phi(T_d)= (T_{\overline{c_1}}\tau T_{\overline{c_2}}\tau T_{\overline{c_3}}\tau T_{\overline{c_4}}\tau)^5.
\]
Again from the chain relation, there exists a separating curve $\overline{d}$ on $S_{h,q}$ such that
\begin{equation}\label{1}
\phi(T_d)=(T_{\overline{c_1}} T_{\overline{c_2}} T_{\overline{c_3}} T_{\overline{c_4}})^5\tau^{20}=T_{\overline{d}}\tau^{20}.
\end{equation}
Since $\overline{c_1},\overline{c_2}, \overline{c_3}, \overline{c_4}, \overline{c_5}$ form a chain, we know that \[i(\overline{c_5},\overline{d})\neq 0.\]
Since $\tau$ commutes with $T_{\overline{c_5}}$, we know \[
\overline{d}\notin \CRS(\tau).\]
Another chain relation gives
\[
T_d=(T_{c_1}T_{c_2}T_{c_3}T_{c_4}T_{c_5'})^6.
\]
By applying $\phi$ to both sides, we know that
\[
\phi(T_d)=(T_{\overline{c_1}}\tau T_{\overline{c_2}}\tau T_{\overline{c_3}}\tau T_{\overline{c_4}}\tau T_{\overline{c_5'}}\tau)^6.
\]
Then there exist two curves $\overline{d_1},\overline{d_2}$ (possibly peripheral) such that 
\begin{equation}\label{2}
\phi(T_d)=(T_{\overline{c_1}} T_{\overline{c_2}} T_{\overline{c_3}} T_{\overline{c_4}}T_{\overline{c_5'}})^6\tau^{30}=T_{\overline{d_1}}T_{\overline{d_2}}\tau^{30}.
\end{equation}
By Equation \eqref{1} and \eqref{2}, we know 
\[
T_{\overline{d_1}}T_{\overline{d_2}}\tau^{10}=T_{\overline{d}}.
\]
Then as $\tau$ commutes with $T_{\overline{d_1}}T_{\overline{d_2}}$, we have 
\[
\{\overline{d}\}=\CRS(T_{\overline{d}})=\CRS(T_{\overline{d_1}}T_{\overline{d_2}}\tau^{10})\subset \{\overline{d_1},\overline{d_2}\}\cup \CRS(\tau)
\]
This means that \[
\overline{d}\subset \{\overline{d_1},\overline{d_2}\},\]
which implies that $\overline{c_1},\overline{c_3},\overline{c_5'}$ bound a pair of pants and one of $\overline{d_1},\overline{d_2}$ is isotopically trivial. Thus we conclude that $\tau^{10}=1$.

We also know that $\tau$ preserves $\overline{c_1},\overline{c_2},\overline{c_3},\overline{c_4},\overline{c_5'}$. Then as a finite order element, $\tau$ is the identity on the pair of pants with boundary $\overline{c_1},\overline{c_3},\overline{c_5'}$. Thus $\tau =1$ by Claim \ref{finiteorder}.
\end{proof}

\bibliographystyle{alpha}
\bibliography{lin2}

\end{document}